\documentclass{amsart}
\usepackage[utf8]{inputenc}

\usepackage[margin=3cm]{geometry}

\usepackage{color}
\usepackage[color,matrix,arrow]{xy}
\input xy
\xyoption{all}

\usepackage{amsmath}
\usepackage{amssymb}
\usepackage{mathrsfs}
\usepackage{faktor}
\usepackage{wasysym}

\usepackage{amsthm}
\theoremstyle{definition} \newtheorem{theorem}{Theorem}[section]
\theoremstyle{definition} \newtheorem{definition}[theorem]{Definition}
\theoremstyle{definition} \newtheorem{lemma}[theorem]{Lemma}
\theoremstyle{definition} \newtheorem{proposition}[theorem]{Proposition}
\theoremstyle{definition} \newtheorem{corollary}[theorem]{Corollary}
\theoremstyle{remark} 
\theoremstyle{remark} \newtheorem{remark}[theorem]{Remark}
\theoremstyle{remark} \newtheorem{example}[theorem]{Example}
\theoremstyle{remark}

\usepackage[style=alphabetic,backend=biber]{biblatex}
\bibliography{references}
\renewbibmacro{in:}{}

\newcommand{\Zpr}{\mathbb{Z}/p^r}
\newcommand{\Zps}{\mathbb{Z}/p^s}
\newcommand{\Zpt}{\mathbb{Z}/p^t}
\newcommand{\Zp}{\mathbb{Z}/p}
\newcommand{\baseInt}{\ell}
\newcommand{\PhiH}{\Phi_H}
\newcommand{\PhiPi}{\Phi_\pi}
\newcommand{\quot}{\theta}
\newcommand{\ev}{\textrm{ev}}
\newcommand{\splitting}{\tau}
\newcommand{\JTens}{\varphi}
\newcommand{\wt}{\textrm{wt}}

\title{$\mathbb{Z}/p^r$-hyperbolicity via homology}
\author{Guy Boyde}
\address{Mathematical Sciences, University of Southampton, Southampton SO17 1BJ, United Kingdom}
\email{gb7g14@soton.ac.uk}

\subjclass[2020]{Primary 55Q05; Secondary 55Q15, 55P40}
\keywords{Local hyperbolicity, Moore spaces}

\begin{document}

\maketitle

\begin{abstract}
We show that the homotopy groups of a Moore space $P^n(p^r)$, where $p^r \neq 2$, are $\mathbb{Z}/p^s$-hyperbolic for $s \leq r$. Combined with work of Huang-Wu, Neisendorfer, and Theriault, this completely resolves the question of when such a Moore space is $\mathbb{Z}/p^s$-hyperbolic for $p \geq 5$, or when $p=2$ and $r \geq 6$. We also give a criterion in ordinary homology for a space to be $\Zpr$-hyperbolic, and deduce some examples.
\end{abstract}

\section{Introduction}

Given a space $X$, one can ask about the behaviour of the partial sum of homotopy groups $$\bigoplus_{i=1}^m \pi_i(X) \textrm{ as } m \rightarrow \infty.$$ Rationally, deep results have been obtained, notably the famous dichotomy of F\'elix, Halperin and Thomas \cite[Chapter 33]{FHT}. Interpreted integrally, this dichotomy says that if $X$ is a simply connected finite $CW$-complex with finite rational category then either \begin{itemize}
    \item the rank of $\bigoplus_{i=1}^\infty \pi_i(X)$ is finite, and $X$ is called \emph{rationally elliptic}, or
    \item the rank of $\bigoplus_{i=1}^m \pi_i(X)$ grows exponentially with $m$, and $X$ is called \emph{rationally hyperbolic}.
\end{itemize} Study of the corresponding behaviour for the torsion parts of these groups, which is the subject of this paper, was initiated by Huang and Wu in \cite{HuangWu}.

Let $M$ be a $\mathbb{Z}$-module, let $p$ be a prime and let $t \in \mathbb{N}$. The \emph{$\Zpt$-dimension} or \emph{$\Zpt$-rank} of $M$, denoted $\dim_{\Zpt}(M)$, is the greatest $d \in \mathbb{N} \cup \{0, \infty \}$ such that there is an isomorphism $M \cong (\Zpt)^d \oplus C$ for some complementary module $C$. Said another way, $\dim_{\Zpt}(M)$ is the number of $\Zpt$-summands in $M$.

\begin{definition} Let $M$ be a graded $\mathbb{Z}$-module, Let $p$ be a prime, and let $S \subset \mathbb{N}$. We say that $X$ is $p$-\textit{hyperbolic concentrated in (the set of exponents)} $S$ if $$a_m := \sum_{t \in S} \dim_{\Zpt}(\bigoplus_{i=1}^m M_i)$$ grows exponentially, in the sense that $$\liminf_m\frac{\ln(a_m)}{m} > 0.$$ For a space $X$ we will say that $X$ is $p$-\textit{hyperbolic concentrated in} $S$ if $\pi_*(X)$ is $p$-hyperbolic concentrated in $S$. If $X$ is $p$-hyperbolic concentrated in $\mathbb{N}$ then we will say simply that $X$ is $p$-hyperbolic. \label{defdef}
\end{definition}

This definition generalises and interpolates between two definitions due to Huang and Wu \cite{HuangWu}. Namely, their \textit{$\Zps$-hyperbolicity} is precisely our $p$-hyperbolicity concentrated in the singleton set $\{s\}$, and their \textit{$p$-hyperbolicity} is precisely our $p$-hyperbolicity concentrated in $\mathbb{N}$, as defined above.

\begin{definition} \label{MooreSpace} Let $P^n(\ell)$ denote the \textit{mod-$\ell$ Moore space}, which we take to be the cofibre $$S^{n-1} \xrightarrow{\ell} S^{n-1} \longrightarrow P^n(\ell)$$ of the degree $\ell$ map. \end{definition}

Huang and Wu show that for  $p$ prime, $n \geq 3$, and $r \geq 1$ the Moore space $P^n(p^r)$ is $\Zpr$ and $\mathbb{Z}/p^{r+1}$-hyperbolic, and additionally that $P^n(2)$ is $\mathbb{Z}/8$-hyperbolic. In \cite{ZhuPan}, Zhu and Pan show that $P^n(p^r)$ is also $\Zp$-hyperbolic. Our first main result fills in the gap between these exponents:

\begin{theorem} \label{MooreHyp} Let $p$ be a prime, and $r \in \mathbb{N}$ with $p^r \neq 2$. If $n \geq 3$, then $P^n(p^r)$ is $\Zps$-hyperbolic for all $s \leq r$. \end{theorem}

The key is to show that the stable homotopy of $P^{n}(p^r)$ contains a $\Zps$-summand for each $s \leq r$. This follows from work of Adams on the $J$-homomorphism \cite{AdamsII, AdamsIV}, which allows us to find such summands in the stable homotopy of spheres, and classical work of Barratt \cite{Barratt} allows us to transplant these summands to Moore spaces. Once this is done, the proof follows the same lines as those in \cite{HuangWu} and \cite{ZhuPan}.

For $p>3$ Huang and Wu's results and Theorem \ref{MooreHyp} together are best possible, in the following sense. In \cite{NeisendorferExponents1}, Neisendorfer shows that $\pi_*(P^n(p^r))$ contains no element of order $p^s$ for $s>r+1$. In fact, Neisendorfer claimed in \cite{NeisendorferExponents1} that this result also holds when $p = 3$, but later, with Brayton Gray, discovered some mistakes in the proof (see the unpublished \cite{NeisendorferExponents2}). These mistakes were repaired apart from when $p=3$. In \cite{NeisendorferExponents2}, Neisendorfer shows that the $3$-primary exponent of $P^n(3^r)$ is either $3^{r+1}$ or $3^{r+2}$.

Neisendorfer's result allows us to combine Huang and Wu's result with Theorem \ref{MooreHyp} to obtain the following (using Proposition \ref{CRT}):

\begin{corollary} \label{FunCor} For $p \neq 2, 3$ prime, $s,\baseInt \in \mathbb{N}$ and $n \geq 3$, the following are equivalent:
\begin{enumerate}
    \item $P^n(\baseInt)$ is $\Zps$-hyperbolic.
    \item $\pi_*(P^n(\baseInt))$ contains a class of order $p^s$.
    \item $p^{s-1} | \baseInt$. \qed
\end{enumerate}
\end{corollary}

Theriault \cite{Theriault2} has shown that for $n \geq 4$ and $r \geq 6$, $\pi_*(P^n(2^r))$ contains no element of order $2^{r+2}$. The result of Corollary \ref{FunCor} therefore holds also when $p=2$ and $\ell$ is divisible by $2^6=64$.

Our second main result is a homological criterion for hyperbolicity:

\begin{theorem} \label{HPrelim} Let $Y$ be a simply connected $CW$-complex, let $p \neq 2$ be prime, and let $s \leq r \in \mathbb{N}$. If there exists a map $$\mu: P^{n+1}(p^r) \longrightarrow Y$$ such that the induced map $$(\Omega \mu)_* : H_*(\Omega P^{n+1}(p^r);\Zps) \longrightarrow H_*(\Omega Y;\Zps)$$ is an injection, then $Y$ is $p$-hyperbolic concentrated in exponents $s, s+1, \dots, r$. In particular if $s=r$ then $Y$ is $\Zpr$-hyperbolic. \end{theorem}

We will see (using Proposition \ref{injection2}) that the hypotheses of Theorem \ref{HPrelim} simplify in the case that $Y=\Sigma X$ is a suspension, as follows:

\begin{theorem} \label{HCriterion} Let $X$ be a connected $CW$-complex, let $p \neq 2$ be prime, and let $s \leq r \in \mathbb{N}$. If there exists a map $$\mu: P^{n+1}(p^r) \longrightarrow \Sigma X$$ such that $$\mu_* : \widetilde{H}_*(P^{n+1}(p^r);\Zps) \longrightarrow \widetilde{H}_*(\Sigma X;\Zps)$$ is an injection, then $\Sigma X$ is $p$-hyperbolic concentrated in exponents $s, s+1, \dots, r$. In particular if $s=r$ then $\Sigma X$ is $\Zpr$-hyperbolic. \end{theorem}

Theorem \ref{HCriterion} is substantially more elementary than existing criteria for $\Zpr$-hyperbolicity: the criterion given in \cite{HuangWu} requires knowledge of a homotopy decomposition of $\Omega Y$, while that of \cite{Me} is given in terms of $K$-theory, and only gives $p$-hyperbolicity. Here, by contrast, we only need ordinary homology.

Together, Theorems \ref{MooreHyp} and \ref{HCriterion} may be thought of as doing for Moore spaces what \cite{Me} did for wedges of spheres. The main difference between the homological results of that paper and this is that the Hurewicz map is enough to detect $p^r$-torsion in the homotopy groups of the Moore space $P^n(p^r)$. In contrast, one needs more sophisticated machinery to see $p^r$-torsion in a wedge of spheres; \cite{Me} used Adams' $e$-invariant. This meant that the theorems of that paper had to be stated in terms of $K$-theory, rather than ordinary homology, and that the spaces under consideration had to be finite complexes.

This document is organized as follows. The proof of Theorem \ref{MooreHyp} may be read independently of the proof of Theorems \ref{HPrelim} and \ref{HCriterion}, and vice versa. Section \ref{ApplicationSection} contains applications of our results. Section \ref{CommonSection} contains definitions needed throughout. Sections \ref{MooreDecompSection}, \ref{CurlySection} and \ref{Proof1Section} prove Theorem \ref{MooreHyp}, while Sections \ref{ModuleSection}, \ref{LieSection}, \ref{FoundationSection} and \ref{proof2Section} prove Theorem \ref{HPrelim}, and Section \ref{SuspensionSection} shows that Theorem \ref{HPrelim} implies Theorem \ref{HCriterion}.

\medskip

I would like to thank my PhD supervisor, Stephen Theriault, for many helpful conversations and much encouragement. Changes to the proof of Theorem \ref{MooreHyp} which make the result work for powers of $2$ are due to him.

\section{Applications} \label{ApplicationSection}

\subsection{Spaces containing a Moore space as a retract}

Various spaces have been shown to contain wedges of Moore spaces and spheres as $p$-local retracts after looping. This section collects some examples of this form.

\begin{example} Let $M$ be an (oriented) $(n-1)$-connected $(2n+1)$-manifold for $n \geq 2$. By Poincar\'e duality, the homology of $M$ is determined entirely by $$H_{n}(M) \cong \mathbb{Z}^r \oplus \bigoplus_{i=1}^\ell \mathbb{Z}/p_i^{r_i}.$$ When $r \geq 1$, Basu \cite[Theorem 5.4]{Basu} gives a decomposition of $\Omega M$, which shows in particular that $\Omega M$ contains a retract $\Omega (\bigvee_{r-1}S^n \vee \bigvee_{r-1} S^{n+1} \vee \bigvee_{i=1}^\ell P^{n}(p_i^{r_i}))$. By Theorem \ref{MooreHyp} and the work of Huang-Wu \cite{HuangWu} and Zhu-Pan \cite{ZhuPan}, it follows that $M$ is $\Zps$-hyperbolic whenever $p^{s-1}$ divides the order of the torsion part of $H_n(M)$. In fact, if $r \geq 2$ then $\Omega M$ contains $\Omega (S^n \vee S^m)$ as a retract, so is $\Zps$-hyperbolic for all $p$ and $s$ by \cite{Me}. Conversely, if $M$ is not $\Zps$ hyperbolic for any $p$ and $s$ (and is not the sphere $S^{2n+1}$) then we must have $H_{n}(M) \cong \mathbb{Z}$. An example of such a manifold is $S^{n-1} \times S^{n}$, whose homotopy groups satisfy $\pi_i(S^{n-1} \times S^{n}) \cong \pi_i(S^{n-1}) \times \pi_i(S^{n})$. Determining hyperbolicity for these examples is therefore as difficult as determining hyperbolicity of $S^n$. \end{example}

In order to use Basu's result, we require that there be a $\mathbb{Z}$-summand in $H_n(M)$. In contrast, our next example has $H_n(M)$ a torsion group.

\begin{example} Let $p$ be an odd prime, let $r \in \mathbb{N}$, and let $M$ be a $5$-dimensional spin manifold with $H_2(M; \mathbb{Z})$ isomorphic to a direct sum of copies of $\Zpr$. In \cite{Theriault} Theriault notes that his Theorem 1.3, together with a classification of simply connected 5-dimensional Poincar\'e duality complexes by St\"ocker \cite{Stocker}, gives a decomposition of $\Omega M$. This decomposition shows that $\Omega M$ contains $\Omega P^{3}(p^r)$ as a retract. In particular, by Theorem \ref{MooreHyp}, $M$ is $\Zps$-hyperbolic for all $1 \leq s \leq r$. \end{example}

\subsection{Suspensions}

This section deduces some examples of Theorem \ref{HCriterion}. As a first example, note that the identity map on the Moore space $P^n(p^r)$ satisfies the hypotheses of that theorem, and so we recover the $s=r$ case of Theorem \ref{MooreHyp} using purely homological methods.

Let $h: \pi_n(Y) \longrightarrow H_n(Y ; \mathbb{Z})$ be the Hurewicz map, which sends a homotopy class $f: S^n \longrightarrow Y$ to the image $f_*(\xi_n)$ of a generator $\xi_n$ of $H_n(S^n ; \mathbb{Z})$ under the map induced on homology by $f$.

\begin{corollary}[of Theorem \ref{HCriterion}]  \label{last} Let $p$ be an odd prime and let $s \in \mathbb{N}$. Suppose that $H_{n-1}(\Sigma X;\mathbb{Z})$ contains a $\Zps$-summand, generated by a class $z \in \textrm{Im}(h)$. Let $\nu : S^{n-1} \longrightarrow \Sigma X$ be a map with $h(\nu) = z$, and let $r \in \mathbb{N}$ be such that the order of $\nu$ is equal to $p^r c$, for $c$ prime to $p$. Then $\Sigma X$ is $p$-hyperbolic concentrated in exponents $s, s+1, \dots , r$. \end{corollary}

Before proving this Corollary, we note that by the Hurewicz Theorem it immediately implies the following.

\begin{corollary} \label{easyH} Let $n$ be the least natural number for which $\widetilde{H}_n(\Sigma X; \mathbb{Z})$ is nontrivial. If $\widetilde{H}_n(\Sigma X; \mathbb{Z})$ contains a $\Zps$-summand, for $p$ an odd prime and $s \in \mathbb{N}$, then $\Sigma X$ is $\Zps$-hyperbolic. \qed \end{corollary}

\begin{proof}[Proof of Corollary \ref{last}] By replacing $\nu$ with $c \nu$ (and $z$ with $cz$) we may assume without loss of generality that $c=1$. Since $\nu$ has order $p^r$, it extends to a map $\mu : P^{n}(p^r) \longrightarrow \Sigma X$.

Let $x$ generate $H_n(P^n(p^r);\Zps)$, and let $y$ generate $H_{n-1}(P^n(p^r);\Zps)$. The Bockstein $\beta$ satisfies $\beta(x)=y$. We have $\mu_*(y) = h(\nu) = z$, and $\beta(\mu_*(x)) = \mu_*(\beta(x)) = \mu_*(y) = z$. This implies that $\mu_*(x)$ and $\mu_*(y)$ must both have order $p^s$, hence that $$ \mu_* : H_*(P^n(p^r);\Zps) \longrightarrow H_*(\Sigma X; \Zps)$$ is an injection. Thus, by Theorem \ref{HCriterion}, $\Sigma X$ is $p$-hyperbolic concentrated in exponents $s, s+1, \dots , r$, as required. \end{proof}

A first example of this sort highlights how much bigger the homotopy of an Eilenberg-MacLane space becomes upon suspending.

\begin{example} The least-dimensional homology of $\Sigma K(\Zps,n)$ is isomorphic to $\Zps$, so Corollary \ref{easyH} implies that $\Sigma K(\Zps,n)$ is $\Zps$-hyperbolic for $p$ odd. \end{example}

More generally, we have:

\begin{example} Let $G$ be a finite group. Atiyah \cite[Theorem 13.1]{Atiyah} has shown that the cohomology of $G$ (which is the cohomology of $K(G,1)$) is nonvanishing in infinitely many degrees. Since the cohomology of $G$ is annihilated by multiplication by $\lvert G \lvert$ \cite[Corollary II.5.4]{AdemMilgram} the lowest-dimensional nontrivial cohomology $H^n(K(G,1); \mathbb{Z})$ must contain a $\Zps$-summand for some $p^s$ dividing $\lvert G \lvert$. By the universal coefficient theorem, the least nontrivial homology is $H_{n-1}(K(G,1); \mathbb{Z})$, which must also contain such a summand. By the suspension isomorphism and Corollary \ref{easyH}, $\Sigma K(G,1)$ is $\Zps$-hyperbolic, provided that $p \neq 2$. In particular, this means that if $\lvert G \lvert$ is odd, then $\Sigma K(G,1)$ is $\Zps$-hyperbolic for some $p^s$ dividing the order of $G$. 

If the (co)homology of $G$ is known in least nontrivial dimension, then we can be more precise. Algebraic interpretations exist for the first few nontrivial homology groups: $H_1(K(G,1),\mathbb{Z})$ is the abelianization $G_{\textrm{ab}}$, and $H_2(K(G,1),\mathbb{Z})$ is known as the \textit{Schur multiplier}. Consider the Alternating groups $A_n$. These are simple, hence have trivial abelianization, and the Schur multiplier is $\mathbb{Z}/2$ unless $n=6,7$, in which case it is $\mathbb{Z}/6$ \cite{Schur}. In particular, Corollary \ref{easyH} implies that the suspended Eilenberg-MacLane spaces of $A_6$ and $A_7$ are $\mathbb{Z}/3$-hyperbolic. Another example is the Suzuki group $\textrm{Suz}$, which is one of the sporadic simple groups, and has Schur Multiplier $\mathbb{Z}/6$ \cite{Griess}, so again $\Sigma K(\textrm{Suz}, 1)$ is $\mathbb{Z}/3$-hyperbolic. \end{example}

\section{Common preamble} \label{CommonSection}

This section collects some foundational material which will be used in the proofs of both main results. First, we have the following well-known proposition, which we use to deduce Corollary \ref{FunCor} from Theorem \ref{MooreHyp}.

\begin{proposition} \label{CRT} Let $n \geq 3$. If $\baseInt \in \mathbb{N}$ has a prime power factorization $\baseInt = p_1^{r_1}p_2^{r_2} \dots p_m^{r_m}$ then $$P^n(\baseInt) \simeq P^n(p_1^{r_1}) \vee P^n(p_2^{r_2}) \vee \dots \vee P^n(p_m^{r_m}),$$ and furthermore $P^n(p^r)$ is $q$-locally contractible for any prime $q \neq p$. \end{proposition}

\begin{proof} Define a map $f: P^n(p_1^{r_1}) \vee P^n(p_2^{r_2}) \vee \dots \vee P^n(p_m^{r_m}) \longrightarrow P^n(\baseInt)$ which is given on the wedge summand $P^n(p_i^{r_i})$ as degree 1 on the top cell and degree $\frac{\baseInt}{p_i^{r_i}}$ on the bottom cell;  that is, according to the following diagram of defining cofibrations. \begin{center}
\begin{tabular}{c}
\xymatrix{
S^{n-1} \ar^{p_i^{r_i}}[r] \ar@{=}[d] & S^{n-1} \ar^{\frac{\baseInt}{p_i^{r_i}}}[d] \ar[r] & P^n(p_i^{r_i}) \ar^{f \lvert_{P^n(p_i^{r_i})}}[d] \\
S^{n-1} \ar^{\baseInt}[r] & S^{n-1} \ar[r] & P^n(\baseInt)
}
\end{tabular}
\end{center}

By the Chinese Remainder Theorem, $f$ induces an isomorphism on integral homology. Thus, by Whitehead's theorem \cite{Whitehead}, $f$ is in fact a homotopy equivalence.

To see that $P^n(p^r)$ is contractible after localization at $q \neq p$, note that the homology with coefficients in the integers localized at $q$, $H_*(P^n(p^r);\mathbb{Z}_{(q)})$, is trivial, and thus by Whitehead's theorem, the inclusion of the basepoint is a homotopy equivalence. \end{proof}

\subsection{The Witt Formula and the Hilton-Milnor Theorem} \label{WittSubs}

We will be interested in counting the dimension of various `weighted components' of free Lie algebras. These Lie algebras will be ungraded in the proof of Theorem \ref{MooreHyp} and will be graded for the proof of Theorems \ref{HPrelim} and \ref{HCriterion}. In both cases, the quantities we wish to count are determined by the Witt formula, which we now define.

Let $\mu: \mathbb{N} \longrightarrow \{-1, 0, 1\}$ be the \textit{M\"obius inversion function}, defined by $$\mu(s) = \begin{cases} 
      1 & s=1 \\
      0 & s>1 \textrm{ is not square free} \\
      (-1)^{\ell} & s>1 \textrm{ is a product of $\ell$ distinct primes.}
      \end{cases}$$
      
The \textit{Witt Formula} $W_n(k)$ is then defined by

$$W_n(k)=\frac{1}{k} \sum_{d \lvert k} \mu(d) n^{\frac{k}{d}}.$$

The Witt formula feeds into the proof of Theorem \ref{MooreHyp} via Theorem \ref{ungradedWitt}, and into the proof of Theorems \ref{HPrelim} and \ref{HCriterion} via Theorem \ref{gradedWittCount}. The asymptotics of the Witt formula are as follows:

\begin{lemma}{\cite[Introduction]{Bahturin}} \label{WittAsymptotics} The ratio
$$\frac{W_n(k)}{\frac{1}{k}n^k}$$
tends to 1 as $k$ tends to $\infty$. \qed
\end{lemma}

We now introduce the Hilton-Milnor Theorem. Let $L$ be the free (ungraded) Lie algebra over $\mathbb{Z}$ on basis elements $x_1 , \dots, x_n$. For an iterated bracket $B$ of the elements $x_i$, let $k_i(B) \in \mathbb{N} \cup \{ 0 \}$ be the number of instances of the generator $x_i$ occurring in $B$. The sum $k(B) = \sum_{i=1}^n k_i(B)$ is called the \emph{weight} of $B$, following Hilton \cite{Hilton}. By induction on $k$, Hilton defines a subset $\mathscr{L}_k$ of the brackets of weight $k$, which he calls the set of \emph{basic products} of weight $k$. The basic products of weight 1 are precisely the $x_i$. The union $\mathscr{L} = \bigcup_{k=1}^\infty \mathscr{L}_k$ is a free basis for $L$ (see for example \cite[Theorem 5.3]{SerreBook}, but note that what we call basic products, Serre calls a \emph{Hall basis}).

\begin{theorem}{\cite[Theorems 3.2, 3.3]{Hilton}} \label{ungradedWitt} Let $L$ be the free Lie algebra over $\mathbb{Z}$ on basis elements $x_1 , \dots, x_n$. Then the cardinality $|\mathscr{L}_k|$ of the set of basic products of weight $k$ is equal to $W_n(k)$. \qed \end{theorem}

We are now ready to state the Hilton-Milnor Theorem. Write $X^{\wedge k }$ for the smash product of $k$ copies of the space $X$.

\begin{theorem}{\cite{Hilton,Milnor}} \label{HiltonMilnor} Let $X_1, X_2, \dots , X_n$ be connected $CW$-complexes. There is a homotopy equivalence $$  \Omega \Sigma ( X_1 \vee \dots \vee X_n) \simeq \prod_{B \in \mathscr{L}} \Omega \Sigma ( X_1^{\wedge k_1(B)} \wedge \dots \wedge X_n^{\wedge k_n(B)} ), $$ where the right hand side is the weak infinite product. \qed \end{theorem}

\section{Decompositions of Moore spaces} \label{MooreDecompSection}

In this section we make the first step in the proof of Theorem \ref{MooreHyp}. Namely, we will see that it follows from work of Cohen, Moore, and Neisendorfer that a Moore space $P^n(p^r)$ with $p^r \neq 2$ contains $P^{n_1}(p^r) \vee P^{n_2}(p^r)$ as a retract after looping, and so it suffices to prove that $P^{n_1}(p^r) \vee P^{n_2}(p^r)$ is $\Zps$-hyperbolic. We will also record Corollary \ref{binomialNeisendorfer}, which describes the behaviour of Moore spaces under iterated smash products.

When $p$ is odd, the loop-decomposition of $P^n(p^r)$ depends on the parity of $n$. We have the following three theorems, which give the three cases of the decomposition.

\begin{theorem}{\cite[Theorem 1.1]{CMNTorsion}} \label{CMNEven} Let $p$ be an odd prime, and let $n>0$. Then \[
\pushQED{\qed} 
\Omega P^{2n+2}(p^r)  \simeq S^{2n+1}\{ p^r \} \times \Omega \bigvee_{m=0}^\infty P^{4n+2mn+3}(p^r).\qedhere
\popQED
\] \end{theorem}

\begin{theorem}{\cite{CMNExponents}} \label{CMNOdd} Let $p$ be an odd prime, and let $n>0$. Then there is a space $T^{2n+1}\{ p^r \}$ so that $$\Omega P^{2n+1}(p^r) \simeq  T^{2n+1}\{ p^r \} \times \Omega \Sigma \bigvee_{\alpha} P^{n_\alpha}(p^r),$$ where $\bigvee_{\alpha} P^{n_\alpha}(p^r)$ is an infinite bouquet of mod-$p^r$ Moore spaces, and each $n_\alpha$ satisfies $n_\alpha \geq 4n-1$. \qed \end{theorem}

\begin{lemma}{\cite[Lemma 2.6]{Cohenp2}} \label{2decomp} Let $n \geq 3$ and $r \geq 2$. Then there exist spaces $T^n\{2^r\}$ such that $$ \Omega P^{n}(2^r) \simeq T^n\{2^r\} \times \Omega \bigvee_{\alpha} P^{m_\alpha}(2^r),$$ where $\bigvee_{\alpha} P^{m_\alpha}(2^r)$ is an infinite bouquet of mod-$2^r$ Moore spaces, and each $m_\alpha$ satisfies $m_\alpha \geq n$. \qed \end{lemma}

Theorems \ref{CMNEven} and \ref{CMNOdd}, together with Lemma \ref{2decomp} immediately imply the following corollary.

\begin{corollary} \label{hasAWedge} Let $p$ be prime and let $r \in \mathbb{N}$. Suppose that $p^r \neq 2$, and let $n \geq 3$. Then $\Omega P^n(p^r)$ has $\Omega ( P^{n_1}(p^r) \vee P^{n_2}(p^r))$ as a retract for some $n_1, n_2 \geq n$. \qed \end{corollary}

Smash powers of Moore spaces are well-understood, by means of the following Lemma.

\begin{lemma}{\cite{NeisendorferMemoir}} \label{Smashing} Let $p$ be prime, and let $r \in \mathbb{N}$, with $p^r \neq 2$. For $n,m \geq 2$, \[
\pushQED{\qed} 
P^n(p^r) \wedge P^m(p^r) \simeq P^{m+n}(p^r) \vee P^{m+n-1}(p^r).\qedhere
\popQED
\]     
\end{lemma}

For a space $X$, write $X^{\vee i}$ for the wedge sum of $i$ copies of $X$. Applying Lemma \ref{Smashing} repeatedly gives the following binomial-type formula.

\begin{corollary} \label{binomialNeisendorfer} Let $p$ be prime, and let $r \in \mathbb{N}$, with $p^r \neq 2$. For $n,m \geq 2$, and $k_1,k_2 \in \mathbb{N}$. Letting $k=k_1+k_2$, we have \[
\pushQED{\qed} 
P^{n}(p^r)^{\wedge k_1} \wedge P^{m}(p^r)^{\wedge k_2} \simeq \bigvee_{i=0}^{k-1}(P^{k_1 n +k_2 m -i}(p^r))^{\vee \binom{k-1}{i}} .\qedhere
\popQED
\] \end{corollary}

\section{Classes in the homotopy groups of $P^n(p^r)$} \label{CurlySection}

In this section, we identify some stable classes in the homotopy groups of $P^n(p^r)$. The identification of these classes is the way in which we go beyond Huang and Wu's work. We will transfer known classes from the stable homotopy groups of spheres (Lemma \ref{sphereClasses}) into the stable homotopy groups of Moore Spaces by means of the stable homotopy exact sequence of the cofibration defining the Moore space. To show that the resulting classes have the correct order, we need assurances about the maximum order of the torsion in the stable homotopy groups of Moore spaces, and these assurances are provided by Corollary \ref{stableMooreAnnihilation}.

Cohen, Moore, and Neisendorfer have shown that the homotopy groups of $P^n(p^r)$ contain classes of order $p^{r+1}$ \cite{CMNTorsion}. However, these classes are all outside the stable range; the stable homotopy groups of $P^n(p^r)$ were already known to be annihilated by multiplication by $p^r$. The proof of this fact is due to Barratt.

\begin{lemma}{\cite{Barratt}} Let $A$ be $(n-1)$-connected, and let $p$ be a prime. Suppose that we have $p^s \textrm{id}_{\Sigma A} \simeq *$ in the group $[\Sigma A, \Sigma A]$, for some $s \in \mathbb{N}$. Then $p^s \pi_{n+j}(\Sigma A) = 0$ for $j \leq (p-1)n$. \label{stableAnnihilation} \qed \end{lemma}

\begin{corollary} Let $p$ be prime, and let $s \in \mathbb{N}$ such that $p^s \neq 2$. Then we have $p^s \pi_{n+j}(P^n(p^s)) = 0$ for $j \leq (p-1)(n-2) - 2$. \label{stableMooreAnnihilation} \end{corollary}

\begin{proof} By definition, $P^n(p^s) \simeq \Sigma P^{n-1}(p^s)$, and $P^{n-1}(p^s)$ is $(n-3)$-connected. By Lemma \ref{stableAnnihilation} the result therefore follows from the fact that the identity map on $P^n(p^s)$ has order $p^s$ \cite[Proposition 6.1.7]{NeisendorferBook}. \end{proof}

We continue in a similar vein. In general, the degree $\ell$ map on $S^n$ does not induce multiplication by $\ell$ on homotopy groups. However, it follows from the Hilton-Milnor Theorem (Theorem \ref{HiltonMilnor}) that it must do so in the stable range, as in the next lemma.

\begin{lemma} \label{stableHM} The degree $\ell$ map $S^n \xrightarrow{\ell} S^n$ induces multiplication by $\ell$ on $\pi_j(S^n)$ for $j \leq 2n-2$. \end{lemma}

\begin{proof} Write $n=m+1$ and $j=i+1$. By the adjoint isomorphism, it suffices to show that $\Omega \ell$ induces multiplication by $\ell$ on $\pi_{i}(\Omega S^{m+1})$ for $i < 2m$. The map $\ell$ is the composition $$ S^{m+1} \xrightarrow{c} \bigvee_{i=1}^\ell S^{m+1} \xrightarrow{\nabla} S^{m+1}$$ of the $\ell$-fold suspension comultiplication $c$ on $S^{m+1}$ with the fold map $\nabla$. Let $\mathscr{L}$ be the free Lie algebra on $\ell$ generators, as in Subsection \ref{WittSubs}. The Hilton-Milnor Theorem (Theorem \ref{HiltonMilnor}) gives a decomposition $$\Omega \bigvee_{i=1}^\ell S^{m+1} \simeq \Omega \prod_{B \in \mathscr{L}} S^{km+1},$$ where $k$ is the weight of $B \in \mathscr{L}$, so in particular is implicitly a function of $B$.

Let $f \in \pi_i(\Omega S^{m+1})$. Applying the above decomposition to $(\Omega \ell)_* (f) = (\Omega \ell) \circ f$ gives factorizations $\varphi$ and $\theta$ as in the following diagram.

\begin{center}
\begin{tabular}{c}
\xymatrix{
\Omega S^{m+1} \ar^{\Omega c}[r] & \Omega \bigvee_{i=1}^\ell S^{m+1} \ar^{\Omega \nabla}[r] & \Omega S^{m+1} \\
S^i \ar^{\varphi}[r] \ar^{f}[u] & \Omega \prod_{B \in \mathscr{L}} S^{km+1}. \ar^{\simeq}[u] \ar^{\theta}[ur] & 
}
\end{tabular}
\end{center}

We must show that $\theta \circ \varphi \simeq \ell f$. Since $i < 2m$, cellular approximation tells us that $\varphi$ factors through the sub-product $\Omega \prod_{i=1}^\ell S^{m+1}$ consisting of those terms where $k=1$. Hilton \cite{Hilton} tells us that the restriction of the Hilton-Milnor map to these summands is given by the product under the loop multiplication of the looped wedge factor inclusions $\Omega S^{m+1} \longrightarrow \Omega \bigvee_{i=1}^\ell S^{m+1}$. Thus, the restriction of $\theta$ to these summands is the $\ell$-fold loop multiplication map $$m : \Omega \prod_{i=1}^\ell S^{m+1} \longrightarrow \Omega S^{m+1}.$$ Furthermore, this restriction of the Hilton-Milnor map is a left homotopy inverse to the looped inclusion $\Omega \iota: \Omega \bigvee_{i=1}^\ell S^{m+1} \longrightarrow \Omega \prod_{i=1}^\ell S^{m+1}$ of the wedge into the product, so $\theta \circ \varphi$ is homotopic to $\theta \circ \Omega \iota \circ \Omega c \circ f$.

To finish, we note that by the axiomatic definition of a comultiplication \cite{Arkowitz} we have that $\Omega \iota \circ \Omega c = \Delta$, the diagonal map into the $\ell$-fold product, and the composition $m \circ \Delta$ is by definition the map inducing multiplication by $\ell$ in the group structure on $[S^i, \Omega S^{m+1}] = \pi_i(\Omega S^{m+1})$ coming from the fact that $\Omega S^{m+1}$ is an $H$-group. But this group structure coincides with that of the homotopy group \cite{Arkowitz}, and so we are done. \end{proof}

Let $\pi^S_j$ denote the $j$-th stable homotopy group of spheres. Work of Adams on the $J$-homomorphism implies that any cyclic group of prime power order occurs as a summand in some $\pi^S_j$:

\begin{lemma}{\cite[Lemma 3.4]{Me}} \label{sphereClasses} For any prime $p$ and any $s \in \mathbb{N}$, there exists $j$ such that $\Zps$ is a direct summand in $\pi^{S}_j$. That is, for a fixed choice of such a $j$, $\Zps$ is a direct summand in $\pi_{n+j}(S^n)$ for all $n \geq j+2$.\qed \end{lemma}

These summands can be transplanted to $P^n(p^r)$ as in the next two corollaries.

\begin{corollary} \label{mooreClasses} Let $p$ be prime, and let $r \geq s \in \mathbb{N}$. If $p^s \neq 2$, then there exists $j$ such that $\Zps$ is a direct summand in $\pi_{n+j}(P^n(p^r))$ for all $n > j+3$. \end{corollary}

 Zhu and Pan \cite{ZhuPan} have already proven the case $s=1$, and Huang and Wu \cite{HuangWu} have already proven the case $s=r$.

\begin{proof} The cofibration $P^n(p^r) \longrightarrow S^n \xrightarrow{p^r} S^n$ gives a truncated long exact sequence on homotopy groups \cite{HiltonBook}: $$ \pi_{2n-3}(P^n(p^r)) \longrightarrow \pi_{2n-3}(S^n) \longrightarrow \pi_{2n-3}(S^n) \longrightarrow \pi_{2n-4}(P^n(p^r)) \longrightarrow \dots $$ $$ \dots \longrightarrow \pi_{n}(P^n(p^r)) \longrightarrow \pi_{n}(S^n) \longrightarrow \pi_{n}(S^n) \longrightarrow \pi_{n-1}(P^n(p^r)) \longrightarrow 0.$$

By Lemma \ref{sphereClasses}, there exists $j$ such that $\Zps$ is a direct summand in $\pi_{n+j}(S^n)$ for all $n \geq j+2$. Fix $n \geq j+4$, and let $f : S^{n+j} \longrightarrow S^n$ generate a $\Zps$-summand. By Lemma \ref{stableHM}, since we are in the stable range, the composite $p^s \circ f$ is homotopic to $p^s f$, and by assumption $f$ has order $p^s$. Thus, since $n \geq j+3$, the exact sequence applies, and taking $r=s$ we obtain a lift $\widetilde{f} \in \pi_{n+j}(P^n(p^s))$ making the following diagram commute.

\begin{center}
\begin{tabular}{c}
\xymatrix{
P^n(p^s) \ar[r] & S^{n} \ar[r]^{p^s} & S^{n}\\
& S^{n + j}. \ar[u]_{f} \ar[ur]_{\simeq *} \ar@{.>}[ul]^{\widetilde{f}}
}
\end{tabular}
\end{center}

We also have, for each $r \geq s$, a diagram \begin{center}
\begin{tabular}{c}
\xymatrix{
S^{n-1} \ar[r]^{p^r} & S^{n-1} \\
S^{n-1} \ar@{=}[u] \ar[r]^{p^s} & S^{n-1}. \ar[u]_{p^{r-s}}
}
\end{tabular}
\end{center}

Extending the rows of this diagram to cofibre sequences and combining with the previous one gives a diagram \begin{center}
\begin{tabular}{c}
\xymatrix{
S^{n-1} \ar[r] & P^n(p^r) \ar^{\underline{\rho}}[r] & S^{n} \ar[r]^{p^r} & S^{n} \\
S^{n-1} \ar[u]_{p^{r-s}} \ar[r] & P^n(p^s) \ar^{\underline{\rho}}[r] \ar[u]_{\varphi} & S^{n} \ar@{=}[u] \ar[r]^{p^s} & S^{n} \ar[u]_{p^{r-s}} \\
& & S^{n+j}, \ar[u]_{f} \ar[ur]_{\simeq *} \ar@{.>}[ul]^{\tilde{f}} &
}
\end{tabular}
\end{center}

We have that $\underline{\rho}_* (\varphi \circ \widetilde{f}) = f$, so the image of $\underline{\rho}_* : \pi_{n+j}(P^n(p^r)) \longrightarrow \pi_{n+j}(S^n)$ contains $f$. Since $f$ generates a $\Zps$-summand, this gives a surjection $\pi_{n+j}(P^n(p^r)) \longrightarrow \Zps$, and it suffices to argue that this surjection is split. From the diagram, it further suffices to do so in the case $r=s$.

By Corollary \ref{stableMooreAnnihilation}, since $n \geq j+4$ we have $p^s \pi_{n+j}(P^n(p^s)) = 0$. This means that the above surjection $\pi_{n+j}(P^n(p^s)) \longrightarrow \Zps$ is a map of $\Zps$-modules with free codomain, so is split, as required. \end{proof}

\begin{corollary} \label{mooreClasses2} Let $r \in \mathbb{N}$. For $n \geq 32$, the group $\pi_{n+28}(P^n(2^r))$ is isomorphic to $\mathbb{Z}/2$. \end{corollary}

This result has already been shown by Zhu and Pan \cite{ZhuPan}, but it is easy to give the more explicit argument below.

\begin{proof} We will take a similar approach to Corollary \ref{mooreClasses}. The argument differs slightly because Corollary \ref{stableMooreAnnihilation} fails when $p^s=2$; we compensate for this using knowledge of the $2$-components of the stable homotopy groups of spheres. Specifically, from \cite[Theorem 1.1.1 and Table 1.1.8]{MahowaldTangora}, we know that the $2$-localization of $\pi^S_{28}$ is isomorphic to $\mathbb{Z}/2$, while the $2$-localization of $\pi^S_{29}$ is trivial.

Let $n \geq 32$. As in the proof of Corollary \ref{mooreClasses}, consider the cofibration $P^n(2^r) \longrightarrow S^n \xrightarrow{2^r} S^n$. The truncated long exact sequence on homotopy groups contains the segment $$\pi_{n+29}(S^n) \longrightarrow \pi_{n+28}(P^n(2^r)) \longrightarrow \pi_{n+28}(S^n).$$ It follows that the $2$-localization of $\pi_{n+28}(P^n(2^r))$ is isomorphic to $\mathbb{Z}/2$. \end{proof}

\section{Proof of Theorem \ref{MooreHyp}} \label{Proof1Section}

In this section, we will prove Theorem \ref{MooreHyp}. In Section \ref{MooreDecompSection}, we reduced the problem to showing $\Zps$-hyperbolicity of the wedge $P^n(p^r) \vee P^m(p^r)$. By the Hilton-Milnor Theorem (Theorem \ref{HiltonMilnor}) and Corollary \ref{binomialNeisendorfer}, we will see that each of the stable classes identified in Section \ref{CurlySection} will give exponentially many summands in the homotopy groups of $P^n(p^r) \vee P^m(p^r)$, which will suffice.

\begin{proof}[Proof of Theorem \ref{MooreHyp}] By Corollary \ref{hasAWedge}, it suffices to prove that if $n,m \geq 2$ then $\Omega(P^{n+1}(p^r) \vee P^{m+1}(p^r))$ is $\Zps$-hyperbolic for all $s \leq r$. Let $\mathscr{L}$ be the free ungraded Lie algebra over $\mathbb{Z}$ on two generators. The Hilton-Milnor theorem (Theorem \ref{HiltonMilnor}) gives $$\Omega (P^{n+1}(p^r) \vee P^{m+1}(p^r)) \simeq \Omega \Sigma (P^{n}(p^r) \vee P^{m}(p^r)) \simeq \prod_{B \in \mathscr{L}} \Omega \Sigma P^{n}(p^r)^{\wedge k_1} \wedge P^{m}(p^r)^{\wedge k_2},$$ where we have written $k_i = k_i(B)$, leaving the fact that $k_i$ is a function of $B$ implicit. Applying Lemma \ref{binomialNeisendorfer} factor-wise, this last is homotopy equivalent to $$\Omega \prod_{B \in \mathscr{L}} \Sigma \bigvee_{i=0}^{k-1}(P^{k_1 n +k_2 m -i}(p^r))^{\vee \binom{k-1}{i}} \simeq \Omega \prod_{B \in \mathscr{L}} \bigvee_{i=0}^{k-1}(P^{k_1 n +k_2 m +1 -i}(p^r))^{\vee \binom{k-1}{i}},$$ where $k=k_1+k_2$ is also implicitly a function of $B$.

By Corollaries \ref{mooreClasses} and \ref{mooreClasses2}, let $j$ be such that $\pi_{N+j}(P^N (p^r))$ contains a $\Zps$-summand for all $N > j+ 3$. For each $B \in \mathscr{L}$, the associated factor of the above decomposition contains $2^{k-1}$ Moore spaces. Supposing without loss of generality that $n \leq m$, the dimensions of these Moore spaces are at least $k(n-1)+2$. Thus, for $k > \frac{j+1}{n-1}$, the homotopy groups of each factor $$\bigvee_{i=0}^{k-1}(P^{k_1 n +k_2 m +1 -i}(p^r))^{\vee \binom{k-1}{i}}$$ contain $2^{k-1}$ summands isomorphic to $\Zps$ in dimensions at most $km+1+j$.

The number of factors for which the weight of $B$ is $k$ is equal to $W_2(k)$ (Theorem \ref{ungradedWitt}), so we may conclude that $$\bigoplus_{i=1}^{km+1+j} \pi_i(P^{n+1}(p^r) \vee P^{m+1}(p^r))$$ contains at least $2^{k-1} W_2(k)$ summands isomorphic to $\Zps$. The sequence $2^{k-1} W_2(k)$ certainly grows exponentially in $k$ (in fact, by Lemma \ref{WittAsymptotics}, it grows like $\frac{1}{2k} 4^k$) and this completes the proof. \end{proof}

\section{Modules over $\Zps$} \label{ModuleSection}

The purpose of this section is to prove various elementary facts about modules over $\Zps$ which we will use later. These facts are mostly intuitively clear, so we recommend that the reader skip this section on first reading, referring back only as necessary.

\subsection{Injections}

The main point of this subsection is to develop the `linear algebra' to prove Lemma \ref{summand}, which says that injections from free $\Zps$-modules are split, and that therefore the `dimension' of the codomain must be at least the `dimension' of the domain.

Let $p$ be prime and let $s \in \mathbb{N}$. Let $M$ be a finitely generated module over $\Zps$. By the structure theorem for finitely generated $\mathbb{Z}$-modules (for example as in \cite[Theorem 7.5]{Lang}) $M$ decomposes as a direct sum $$M \cong \bigoplus_{i=1}^n \mathbb{Z}/p^{s_i},$$ where each $s_i$ satisfies $1 \leq s_i \leq s$. Further, if we order the summands so that $s_{i+1} \geq s_i$, then the sequence $(s_i \ \lvert \ 1 \leq i \leq n)$ is uniquely determined. In particular, if we fix $t \in \mathbb{N}$, then the number of values of $i$ for which $s_i = t$ is uniquely determined. This number is then precisely the \emph{$\Zpt$-dimension} $\dim_{\Zpt}(M)$ of Definition \ref{defdef}. We will often use without comment the fact that a $\Zps$-module is equivalently a $\mathbb{Z}$-module $M$ satisfying $p^sM = 0$.

We will wish to mimic the approach of ordinary linear algebra as far as possible. We will wish to be able to `change basis', and to do so we need a notion of basis, which must generalize the idea of a free basis in that our elements may have variable order.

\begin{definition} Let $M$ be a $\Zps$-module. A \emph{basis} of $M$ is a list $$((e_i,s_i) \in M \times \mathbb{N} \ \lvert \ 1 \leq i \leq n ),$$ such that the following conditions are satisfied:
\begin{itemize}
    \item Each $x \in M$ is expressible as $x = \sum_{i=1}^n \lambda_i e_i$ for $\lambda_i \in \Zps$ (\emph{spanning}).
    \item $\sum_{i=1}^n \lambda_i e_i = 0$ if and only if $p^{s_i} \lvert \lambda_i$ for each $i$ (\emph{linear independence}).
\end{itemize}
\end{definition}

\begin{lemma} Any finitely generated $\Zps$-module has a basis. Conversely, if $((e_i,s_i) \ \lvert \ 1 \leq i \leq n )$ is a basis of $M$, then the map $$\bigoplus_{i=1}^n \mathbb{Z}/p^{s_i} \longrightarrow M$$ defined by sending the generator of the $i$-th summand to $e_i$ is an isomorphism.
\end{lemma}

\begin{proof} To see that $M$ has a basis write $M \cong \bigoplus_{i=1}^n \mathbb{Z}/p^{t_i}$, taking $e_i$ to be a generator of the $i$-th summand, and taking $s_i = t_i$. It follows immediately that this is a basis.

Conversely, let $\varphi: \bigoplus_{i=1}^n \mathbb{Z}/p^{s_i} \longrightarrow M$ be as in the theorem statement. By linear independence of the basis, $p^{s_i} e_i = 0$ for each $i$, so $\varphi$ is well-defined. Surjectivity of $\varphi$ follows immediately from the spanning condition, while injectivity follows immediately from linear independence. Thus, $\varphi$ is an isomorphism, as required. \end{proof}

\begin{lemma} \label{basisManeuvers} Let $((e_i,s_i) \ \lvert \ 1 \leq i \leq n )$ be a basis of $M$.
\begin{itemize}
    \item If $\lambda$ is a unit in $\Zps$, then replacing the basis element $(e_k,s_k)$ with $(\lambda e_k,s_k)$ again yields a basis.
    \item If $j \neq k$ and $s_j \leq s_k$, then replacing the basis element $(e_k,s_k)$ with $(e_k + \mu e_j, s_k)$ for any $\mu \in \Zps$ again yields a basis.
\end{itemize}
\end{lemma}

\begin{proof} We will show only that the basis obtained by the second replacement is linearly independent; the other parts are similar.

Write $(e_i', s_i)$ for the new basis, and suppose that $\sum_{i=1}^n \lambda_i e_i' = 0$. We must show that $p^{s_i}$ divides $\lambda_i$ for each $i$. Substituting in, we have $(\sum_{i \neq j,k} \lambda_i e_i) + \lambda_j e_j + \lambda_k (e_k + \mu e_j) = 0$. Since the original basis was linearly independent, we have that $p^{s_i} \lvert \lambda_i$ for $i \neq j$. In particular, $p^{s_k} \lvert \lambda_k$. We also have $p^{s_j} \lvert (\lambda_j + \mu \lambda_k)$. Since $s_j \leq s_k$ we have $p^{s_j} \lvert \lambda_k$, so $p^{s_j} \lvert \lambda_j$. Thus, $p^{s_i} \lvert \lambda_i$ for all $i$, and thus the $(e_i',s_i)$ form a basis, as required. \end{proof}

It is always true that a surjection onto a free module splits; over $\Zps$, it is additionally true that an injection from a free module splits.

\begin{lemma} \label{summand} Let $M$ and $N$ be finitely-generated $\Zps$-modules, with $M$ free. The image of any injection of $\Zps$-modules $\varphi : M \longrightarrow N$ is a summand, and $\dim_{\Zps}(N) \geq \dim_{\Zps}(M)$. \end{lemma}

\begin{proof} Let $(x_1,t_1), \dots, (x_m,t_m)$ be a basis of $M$, and let $$(e_1,s_1), \dots , (e_n,s_n), (e'_1,s_1'), \dots (e'_{n'},s_{n'}')$$ be a basis of $N$, such that each $s_i=s$ and each $s'_i<s$.

Thus we have $f(x_1) = \sum_{i=1}^n \lambda_{i} e_i + \sum_{i=1}^{n'} \lambda_{i}' e_i'$ for some coefficients $\lambda_i$ and $\lambda_i'$. In particular, since $f(x_1)$ has order $p^s$, there must be some $\lambda_i$ which is not divisible by $p$.By repeated use of Lemma \ref{basisManeuvers} we may therefore change basis in $M$ by replacing $e_i$ by $\sum_{i=1}^n \lambda_{i} e_i + \sum_{i=1}^{n'} \lambda_{i}' e_i'$. After this change we have $f(x_1) = e_i$, and by renumbering we may assume that $i=1$.

We repeat this procedure inductively: at the $j$-th stage we have $f(x_i) = e_i$ for all $i < j$ and we wish to arrange that $f(x_j) = e_j$. We have that $f(x_j) = \sum_{i=1}^n \lambda_{i} e_i + \sum_{i=1}^{n'} \lambda_{i}' e_i'$ for some coefficients $\lambda_i$ and $\lambda_i'$, and the set $f(x_1), \dots , f(x_{j-1})$ spans the submodule $\langle e_1, \dots , e_{j-1} \rangle \subset M$. By changing basis according to Lemma \ref{basisManeuvers}, we may arrange that $\lambda_i = 0$ for $i<j$, and this does not change the fact that $f(x_i) = e_i$ for these values of $i$. Again, $f(x_j)$ has order $p^s$, so there must be $i \geq j$ with $\lambda_i$ not divisible by $p$, and by renumbering we may assume that $i=j$. By changing basis we may arrange that $f(x_j) = e_j$. This completes the inductive step, hence the proof that $\textrm{Im}(f)$ is a summand. Since after this procedure we have $f(x_i) = e_i$ for $i = 1, \dots , m$ we must have $n \geq m$, which is the other part of the theorem statement. \end{proof}

We also have the following technical lemma, which will be used in the proof of Proposition \ref{injection2}.

\begin{lemma} \label{factorTensor} Let $X$, $A$, $B$, and $Y$ be $\Zps$-modules, with $X$ free and $p^{s-1} B = 0$. Let $f: X \longrightarrow A \oplus B$ and $g: A \oplus B \longrightarrow Y$ be homomorphisms. Let $i_A$ be the inclusion of $A$ in $A \oplus B$, and let $\pi_A$ be the projection $A \oplus B \longrightarrow A$. If $g \circ f$ is injective, then the composite $g \circ i_A \circ \pi_A \circ f$ is also injective.
\end{lemma}

\begin{proof} Since $X$ is free, a map defined on $X$ is an injection if and only if its restriction to $p^{s-1}X$ is an injection. It therefore suffices to show that if $g \circ i_A \circ \pi_A \circ f (p^{s-1}x) = 0$ then $p^{s-1} x = 0$.

Thus, suppose that $g \circ i_A \circ \pi_A \circ f (p^{s-1}x) = 0$. Write $f(x) = a+b \in A \oplus B$, for $a \in A$ and $b \in B$. Then $f(p^{s-1}x) = p^{s-1}a$, since $p^{s-1}B=0$. In particular, $f(p^{s-1}x) = i_A \circ \pi_A \circ f(p^{s-1}x)$. Thus, $g \circ f (p^{s-1}x) = 0$, and $g \circ f$ is an injection, so $p^{s-1}x=0$, as required. \end{proof}

\subsection{Surjections}

The main result of this subsection is Lemma \ref{surjectOnHype}, which is the basic algebraic scaffolding for the proof of Theorem \ref{HPrelim}.

\begin{lemma} \label{sCase} Let $\varphi: M \longrightarrow N$ be a surjection of $\Zps$-modules. Then $$ \dim_{\Zps}(M) \geq \dim_{\Zps}(N).$$ \end{lemma}

\begin{proof} Write $N = F \oplus C$, where $F$ is free over $\Zps$, and the complementary module $C$ satisfies $p^{s-1}C=0$. Let $\pi: N \longrightarrow F$ be the projection. The map $\pi \circ \varphi$ is a composite of surjections, hence a surjection, so is split by freeness of $F$. Thus, we have an isomorphism $M \cong F \oplus D$ for some complementary module $D$, so $$\dim_{\Zps}(M) \geq \dim_{\Zps}(F) = \dim_{\Zps}(N),$$ as required. \end{proof}

\begin{lemma} \label{downDog} Let $A$ be a submodule of a $\Zps$-module $N$, such that $A+pN=N$. Then $A=N$.
\end{lemma}

\begin{proof} Because $N$ is a $\Zps$-module, we have $p^s N = 0$, so certainly $A \supset p^s N$. We will now show that if $A \supset p^k N$ then $A \supset p^{k-1} N$. By induction, this implies that $A \supset p^0 N = N$, which suffices.

Assume that $A \supset p^k N$, and let $z \in N$. We have by assumption that $z = x + py$ for $x \in A$ and $y \in N$. Thus, $p^{k-1} z = p^{k-1} x + p^k y$. But now, $p^k y \in p^k N$, which by induction is a subset of $A$, so $p^{k-1} z \in A$, and since $z$ is an arbitrary element of $N$, this implies that $p^{k-1} N \subset A$. This completes the inductive step, hence the proof. \end{proof}

\begin{lemma} \label{realSurjection} Let $M,M',N$ be $\mathbb{Z}$-modules. Let $p$ be prime and let $s \leq r \in \mathbb{N}$. Suppose that $p^r  M = 0$, so $M$ may be regarded as a module over $\Zpr$, and that $p^s N = 0$. Let $\varphi: M \longrightarrow N$ be a surjection which admits a factorization \begin{center}
\begin{tabular}{c}
\xymatrix{
M' \ar^{\widetilde{\varphi}}[dr] & \\
M \ar^{\varphi}[r] \ar^{\iota}[u] & N.
}
\end{tabular}
\end{center}

Then $\sum_{t=s}^r \dim_{\Zpt}(M') \geq \dim_{\Zps}(N).$
\end{lemma}

\begin{proof} We will first argue that we may assume $p^r M'=0$ without loss of generality. Write $M' = A \oplus B$, where $p^r A = 0$, and $B$ is a direct sum of copies of $\mathbb{Z}$, $\mathbb{Z}/q^t$ for various $q \neq p$ and $t \in \mathbb{N}$, and $\Zpt$ for $t>r$. This gives a decomposition $M = \iota^{-1}(A) \oplus \iota^{-1}(B)$. The restriction of $\iota$ to $\iota^{-1}(B)$ must have image contained in $pB$, so the same restriction of $\widetilde{\varphi} \circ \iota$ has image contained in $pN$. Furthermore, since $\widetilde{\varphi}$ is a surjection, we have that $\textrm{Im}(\widetilde{\varphi} \circ \iota \lvert_{\iota^{-1}(A)}) + \textrm{Im}(\widetilde{\varphi} \circ \iota \lvert_{\iota^{-1}(B)}) = N$, so in particular $\textrm{Im}(\widetilde{\varphi} \circ \iota \lvert_{\iota^{-1}(A)}) + pN = N$. By Lemma \ref{downDog} we then have $\textrm{Im}(\widetilde{\varphi} \circ \iota \lvert_{\iota^{-1}(A)}) = N$. We may therefore restrict $M'$ to $A$ and $M$ to $\iota^{-1}(A)$ in the diagram without affecting the hypotheses. In particular, since $p^{r}A=0$ it suffices to prove the lemma in the case that $p^r M'=0$.

We now tensor the diagram with $\Zps$; since $p^s N=0$, we have $N \otimes \Zps \cong N$. Since $p^r M' = 0$, we have $\dim_{\Zps}(M' \otimes \Zps) = \sum_{t=s}^r \dim_{\Zpt}(M')$. By Lemma \ref{sCase}, since $\widetilde{\varphi} \otimes \Zps$ is a surjection we have $\dim_{\Zps}(M' \otimes \Zps) \geq \dim_{\Zps}(N \otimes \Zps)$, which completes the proof. \end{proof}

By applying Lemma \ref{realSurjection} in each degree we immediately obtain the following.

\begin{corollary}[The `Sandwich' Lemma] \label{surjectOnHype} Let $M,M',N$ be graded $\mathbb{Z}$-modules. Let $p$ be prime and let $r\geq s \in \mathbb{N}$. Suppose that $p^r  M = 0$ and that $p^s N = 0$. Let $\varphi: M \longrightarrow N$ be a surjection which admits a factorization \begin{center}
\begin{tabular}{c}
\xymatrix{
M' \ar[dr] & \\
M \ar^{\varphi}[r] \ar[u] & N.
}
\end{tabular}
\end{center}

If $N$ is $\Zps$-hyperbolic then $M'$ is $p$-hyperbolic concentrated in exponents $s, s+1, \dots, r$. \qed \end{corollary}

\begin{lemma} \label{reduceToField} Let $\varphi: M \longrightarrow N$ be a map of $\Zps$-modules, with $N$ free. Then $\dim_{\Zps}(\textrm{Im}(\varphi)) = \dim_{\Zp}(\textrm{Im}(\varphi \otimes \Zp))$. \end{lemma}

\begin{proof} Let $(e_1,s_1), \dots , (e_m,s_m), (e_1',s_1'), \dots , (e_{m'}',s_{m'}')$ be a basis of $M$, where $s_i = s$ and $s_i' < s$. Let $S$ be a maximal subset of the $e_i$ such that the restriction of $\varphi$ to the submodule of $M$ generated by $S$ is an injection. Denote this submodule by $\langle S \rangle$. By renumbering we may assume that $S=\{e_1, \dots e_k \}$ for some $k \leq n$. We clearly have $\textrm{Im}(\varphi \lvert_{\langle S \rangle}) \subset \textrm{Im}(\varphi)$, and we will now show that $\textrm{Im}(\varphi) \subset \textrm{Im}(\varphi \lvert_{\langle S \rangle}) + pN$.

Since $N$ is assumed free, and the elements $e_i'$ have order $p^{s_i}$ for $s_i < s$, we must have $\varphi(e_i') \in pN$. Now consider $e_j$, for $k+1 \leq j \leq m$. By construction of $S$, the restriction of $\varphi$ to $\langle S \cup \{ e_j \} \rangle$ is not injective, so there exist $\lambda_1 , \dots \lambda_k, \lambda \in \Zps$ with $\lambda \neq 0$ such that $\varphi(\sum_{i=1}^k \lambda_i e_i + \lambda e_j) = 0$. This implies that $\lambda \varphi(e_j) \in \textrm{Im}( \varphi \lvert_{\langle S \rangle})$. Thus,  $p^t \varphi(e_j) \in \textrm{Im}(\varphi \lvert_{\langle S \rangle})$ for some $t < s$. By Lemma \ref{summand} we may write $N = \textrm{Im}(\varphi \lvert_{\langle S \rangle}) \oplus C$ for some complementary module $C$, and under this correspondence we have $\varphi (e_j) = (\beta, \gamma)$ for $\gamma \in C$ and $\beta \in \textrm{Im}(\varphi \lvert_{\langle S \rangle})$. Since $p^t \varphi(e_j) \in \textrm{Im}(\varphi \lvert_{\langle S \rangle})$, we have $p^t \gamma = 0$, so by freeness of $N$, $t<s$ implies that $\gamma \in pN$, so $\varphi(e_j) \in \textrm{Im}(\varphi\lvert_{\langle S \rangle}) + pN$. We have now shown that all elements of the basis of $M$ are carried under $\varphi$ to $\textrm{Im}(\varphi\lvert_{\langle S \rangle}) + pN$, so $\textrm{Im}(\varphi) \subset \textrm{Im}(\varphi\lvert_{\langle S \rangle}) + pN$, as claimed.

Now, $\varphi\lvert_{\langle S \rangle}$ is split by Lemma \ref{summand}, so $\dim_{\Zps}(\textrm{Im}(\varphi\lvert_{\langle S \rangle})) = k$. Furthermore, by taking the inclusion on each summand there is a surjection $\textrm{Im}(\varphi\lvert_{\langle S \rangle}) \oplus pN \longrightarrow \textrm{Im}(\varphi\lvert_{\langle S \rangle}) + pN \subset N$, and $pN$ is annihilated by multiplication by $p^{s-1}$, so by Lemma \ref{sCase} $\dim_{\Zps}(\textrm{Im}(\varphi\lvert_{\langle S \rangle}) + pN) \leq k$. Since $\dim_{\Zps}(\textrm{Im}(\varphi\lvert_{\langle S \rangle}) + pN) \geq \dim_{\Zps}(\textrm{Im}(\varphi\lvert_{\langle S \rangle}))$, this implies that the former is equal to $k$. Thus, since $$\textrm{Im}(\varphi\lvert_{\langle S \rangle}) \subset \textrm{Im}(\varphi) \subset \textrm{Im}(\varphi\lvert_{\langle S \rangle}) + pN$$ applying Lemma \ref{summand} to the inclusions gives $$k= \dim_{\Zps}(\textrm{Im}(\varphi\lvert_{\langle S \rangle})) \leq \dim_{\Zps}(\textrm{Im}(\varphi)) \leq \dim_{\Zps}(\textrm{Im}(\varphi\lvert_{\langle S \rangle})+pN) = k,$$ so $\dim_{\Zps}(\textrm{Im}(\varphi))=k$.

To finish the proof we must show that $\dim_{\Zp}(\textrm{Im}(\varphi \otimes \Zp)) = k$. Since the images of $\varphi$ and $\varphi \lvert_{\langle S \rangle}$ differ only by at most $pN$, we have $\textrm{Im}(\varphi \otimes \Zp) = \textrm{Im}(\varphi \lvert_{\langle S \rangle} \otimes \Zp)$. Since $\varphi \lvert_{\langle S \rangle}$ is split injective, $\varphi \lvert_{\langle S \rangle} \otimes \Zp$ is injective, so $\dim_{\Zp}(\varphi \lvert_{\langle S \rangle} \otimes \Zp) = k$, which completes the proof. \end{proof}

\subsection{\textrm{Tor} and the Universal Coefficient Theorem}

The purpose of this section is to prove that for $t<s$ a map inducing an injection on homology with $\Zps$-coefficients also induces an injection on homology with $\Zpt$-coefficients (Lemma \ref{toTheFields}) provided that the domain is free. This follows straightforwardly from the Universal Coefficient Theorem for homology, where we regard $\Zpt$ as a module over $\Zps$. The inclusion of the bottom cell of a Moore space provides an easy counterexample to the converse; the algebraic point being that the converse of Lemma \ref{oneway} is false.

\begin{lemma} \label{TorAnnihilation} For any finitely generated $\Zps$-modules $M$, $N$ we have $$p^{s-1}\textrm{Tor}_{\Zps}(M,N)=0,$$ and furthermore if $M$ or $N$ is free then $\textrm{Tor}_{\Zps}(M,N)=0$. \end{lemma}

\begin{proof} For any ring $R$ and $R$-module $M$ we have $\textrm{Tor}_R(R,M)=0$, since $R$ is free as an $R$-module. If $1 \leq t<s$, then a free resolution of $\Zpt$ over $\Zps$ is given by $$0 \longrightarrow \Zps \xrightarrow{\cdot p^t} \Zps \longrightarrow 0,$$ so,  for any $\Zps$-module $M$, $\textrm{Tor}_{\Zps}(\Zpt, M) = \textrm{Ker}(M \xrightarrow{\cdot p^s} M)$, which is annihilated by multiplication by $p^t$, hence in particular is annihilated by multiplication by $p^{s-1}$. Since any $\Zps$-module decomposes as a direct sum of modules isomorphic to $\Zpt$ for $1 \leq t \leq s$, both parts of the Lemma now follow by additivity of $\textrm{Tor}$. \end{proof}

\begin{lemma} \label{oneway} Let $\varphi: M \longrightarrow N$ be a map of $\Zps$-modules, with $M$ free. Let $t<s$. If $\varphi$ is injective then $\varphi \otimes \Zpt : M \otimes \Zpt \longrightarrow N \otimes \Zpt$ is injective. \end{lemma}

\begin{proof} Note that $M \otimes \Zpt$ is a free $\Zpt$-module. Suppose that $\varphi \otimes \Zpt$ is not injective. Then there exists $x\in M$ which is not divisible by $p^t$ such that $\varphi(x)$ is divisible by $p^t$. By freeness of $M$, $p^{s-t} x$ is not divisible by $p^s$, hence is nonzero, but $\varphi(p^{s-t} x) = p^{s-t} \varphi (x)$ is divisible by $p^s$, hence is zero. That is, $\varphi$ is not injective. \end{proof}

\begin{lemma} \label{toTheFields} Let $t<s \in \mathbb{N}$. Let $f: X \longrightarrow Y$ be a map of spaces, and suppose that $H_*(X ; \Zps)$ is a free $\Zps$-module. If $$f_* : H_*(X ; \Zps) \longrightarrow H_*(Y ; \Zps)$$ is injective then $$f_* : H_*(X ; \Zpt) \longrightarrow H_*(Y ; \Zpt)$$ is injective. \end{lemma}

\begin{proof} Write $f_*^t$ for the induced map on homology with $\Zpt$-coefficients, and likewise $f_*^s$. Applying the universal coefficient theorem for the module $\Zpt$ over the ring $\Zps$ we get a map of short exact sequences \begin{center}
\begin{tabular}{c}
\xymatrix{
0 \ar[r] & H_n(X; \Zps) \otimes \Zpt \ar[r] \ar^{f^s_* \otimes \Zpt}[d] & H_n(X; \Zpt) \ar[r] \ar^{f^t_*}[d] & 0 \ar[d] \ar[r] & 0 \\
0 \ar[r] & H_n(Y; \Zps) \otimes \Zpt \ar[r] & H_n(Y; \Zpt) \ar[r] & \textrm{Tor}(H_{n-1}(Y;\Zps),\Zpt) \ar[r] & 0.
}
\end{tabular}
\end{center}

The $\textrm{Tor}$ term in the top row vanishes by the freeness hypothesis on $H_*(X; \Zps)$. Since the first map in each exact sequence is an injection, $f^t_*$ is injective if and only if $f^s_* \otimes \Zpt$ is injective. By Lemma \ref{oneway}, if $f^s_*$ is injective, then $f^s_* \otimes \Zpt$ is injective, so $f^t_*$ is injective, as required. \end{proof}

\section{Free differential Lie algebras} \label{LieSection}

In this section we will show that the module of boundaries  $BL(x,dx)$ in the free differential Lie algebra $L(x,dx)$ over $\Zpr$ is $\Zpr$-hyperbolic. In the situation of Theorem \ref{HPrelim} we will obtain a factorization of the tensor map $$BL(x,dx) \longrightarrow \pi_*(\Omega Y) \longrightarrow BL(x,dx) \otimes \Zps,$$ which will imply by Corollary \ref{surjectOnHype} (The `Sandwich' Lemma) that $\Omega Y$ must be $p$-hyperbolic concentrated in exponents $s, s+1, \dots , r$. The desired $\Zpr$-hyperbolicity of $BL(x,dx)$ will follow from Cohen, Moore, and Neisendorfer's description of the homology of $L(x,dx)$, which is Proposition \ref{acyclicHomology}.

Throughout this section we work over a ground ring $R = \Zpr$ for $p \neq 2$. The next definitions are as in \cite{CMNTorsion}.

\begin{definition} \label{LieDef} A \textit{graded Lie algebra} is a graded $\Zpr$-module $L$, together with a $\Zpr$-bilinear pairing $$[ \phantom{x},\phantom{x} ]: L_n \times L_m \longrightarrow L_{n+m},$$ called a \textit{Lie bracket} which satisfies the relations of \begin{itemize}
    \item (antisymmetry): $[x,y]=-(-1)^{\textrm{deg}(x)\textrm{deg}(y)}[y,x]$ for all $x$ and $y$ in $L$.
    \item (the Jacobi identity): $[x,[y,z]] = [[x,y],z]+(-1)^{\textrm{deg}(x)\textrm{deg}(y)}[y,[x,z]]$ for all $x$, $y$, and $z$ in $L$.
    \item $[x,[x,x]]=0$ for all $x$ of odd degree.
\end{itemize}  \end{definition}

Let $V$ be a graded $\Zpr$-module. Denote by $L(V)$ the \textit{free Lie algebra} on $V$. There is a linear map $j: V \longrightarrow L(V)$ and $L(V)$ is characterized up to canonical isomorphism as follows. For any map $f: V \longrightarrow L$ where $L$ is a graded Lie algebra, there is a unique map $g:L(V) \longrightarrow L$ so that $g \circ j=f$. The Lie algebra $L(V)$ may be constructed as follows.

Let $L'(V)$ be the free nonassociative graded algebra on $V$, where we think of the operation as a bracket. Precisely, let $B_k$ be the set of bracketings of a string of $k$ symbols. Concatenation of bracketings gives an operation $B_{k_1} \times B_{k_2} \longrightarrow B_{k_1 + k_2}$, which makes $B = \bigcup_{i=1}^\infty B_i$ into a magma. As a module, $$L'(V)=\bigoplus_{k=1}^\infty (\bigoplus_{b \in B_k} V^{\otimes k}),$$ where we think of each copy of $V^{\otimes k}$ as being bracketed according to $b$. The bracket operation on $L'(V)$ is obtained by extending the operation on $B$ bilinearly.

The free Lie algebra $L(V)$ is obtained as the quotient of $L'(V)$ by the relations of Definition \ref{LieDef}, and automatically has the desired universal property. Denote by $\quot$ the quotient map $L'(V) \longrightarrow L(V)$. It follows that for $s<r$, we have $L(V \otimes \Zps)=L(V) \otimes \Zps$.

Note also that any map from $V$ into a graded $\Zpr$-module $A$ with a bilinear operation (that is to say, a nonassociative $\Zpr$-algebra) extends uniquely to a map of graded nonassociative algebras $L'(V) \longrightarrow  A$. The map $\quot$ is a map of nonassociative algebras, hence is uniquely determined by its effect on $V$, and we may therefore call it the \emph{natural quotient}.

\begin{definition} A \textit{differential Lie algebra} is a graded Lie algebra together with an $\Zpr$-linear map $d:L \longrightarrow L$ of degree $-1$, which \begin{itemize}
    \item is a differential: $d^2(x)=0$ for all $x$ in $L$.
    \item is a derivation: $d[x,y]=[dx,y]+(-1)^{\textrm{deg}(x)}[x,dy]$ for all $x$ and $y$ in $L$.
\end{itemize}
\end{definition}

If $V$ carries a differential $d$, then we may define a differential on $L'(V)$ which is the unique derivation extending $d$. This differential can be seen to satisfy the relations of Definition \ref{LieDef}, and therefore descends to give a differential on $L(V)$, which makes $L(V)$ into a differential Lie algebra.

When $p=3$, Samelson products in $\pi_*(\Omega X; \mathbb{Z}/3^r)$ fail to satisfy the Jacobi identity, so $L'(V)$ will also serve as a version of $L(V)$ which does not satisfy the Jacobi identity. For $p \neq 3$, $L'(V)$ may be replaced with $L(V)$ everywhere in this paper, which slightly simplifies things \cite[Remark 6.3]{CMNTorsion}.

Write $L(V)^k$ for the weight-$k$ component of $L(V)$, that is, the submodule generated by brackets of length $k$ in the elements of $V$. It follows from our construction of the free Lie algebra $L(V)$ that $L(V) \cong \bigoplus_{k=1}^\infty L(V)^k$, so weight gives a second grading on $L(V)$, and we shall write $\wt(x)=k$ whenever $x \in L(V)^k$. We will use subscripts (as in $L(V)_i$) for ordinary grading, and superscripts (as in $L(V)^k$) for weight. The dimension of the weighted components is given by the Witt formula, which we defined in Section \ref{CommonSection}.

\begin{theorem}{\cite[Theorems 3.2, 3.3]{Hilton}} \label{gradedWittCount} Let $V$ be a free graded $\mathbb{Z}$- or $\Zps$-module of total dimension $n$. Then the total dimension of $L(V)^k$ is $W_n(k)$. \qed \end{theorem}

\subsection{Homology and boundaries}

Let $x$ be an even-dimensional class in a graded Lie algebra $L$ over $\Zpr$ for $p \neq 2$. Let $$\tau_k(x) = \textrm{ad}^{p^k-1}(x)(dx),$$  $$\textrm{so } \deg(\tau_k(x))=p^k\deg(x)-1$$ and let $$\sigma_k(x) = \frac{1}{2} \sum_{j=1}^{p^k-1} \frac{1}{p}{p^k \choose j}[\textrm{ad}^{j-1}(x)(dx), \textrm{ad}^{p^k-1-j}(x)(dx)],$$ $$\textrm{so } \deg(\sigma_k(x))=p^k\deg(x)-2,$$ where we understand the coefficients $\frac{1}{p}{p^k \choose j}$ to be computed in the integers and then reduced mod $p$.

\begin{proposition}{\cite[Proposition 4.9]{CMNTorsion}} \label{acyclicHomology} Let $V$ be an acyclic differential $\Zp$-vector space. Write $L(V) \cong HL(V) \oplus K$, for an acyclic module $K$. If $K$ has an acyclic basis, that is, a basis $$\{x_\alpha, y_\alpha, z_\beta, w_\beta\},$$ where $\alpha$ and $\beta$ range over index sets $\mathscr{I}$ and $\mathscr{J}$ respectively, and we have $$d(x_\alpha) = y_\alpha, \textrm{ deg}(x_\alpha) \textrm{ even,}$$ $$d(z_\beta) = w_\beta, \textrm{ deg}(z_\beta) \textrm{ odd,}$$ then $HL(V)$ has a basis $$\{\tau_k(x_\alpha), \sigma_k(x_\alpha)\}_{\alpha \in \mathscr{I}, k \geq 1}.$$ \end{proposition}

\begin{remark} \label{basisChoice} An acyclic basis for $K$ may always be chosen, by the following inductive procedure. Write $K_i$ for the $i$-th graded component of $K$. Then $d: K_{i+1} \longrightarrow K_i$, and since $K$ is acyclic we have $\textrm{Im}(d)=\textrm{Ker}(d)$ in each $K_i$. Assume that we have a basis of $\textrm{Ker}(d) \subset K_i$. Because $\textrm{Ker}(d) = \textrm{Im}(d)$, $d$ induces an isomorphism $\faktor{K_{i+1}}{\textrm{Ker}(d)} \longrightarrow \textrm{Im}(d)$. Choose representatives of this basis in $K_{i+1}$, and choose a basis of $\textrm{Ker}(d) \subset K_{i+1}$. Combining these two sets gives a basis of $K_{i+1}$, and the subset which forms a basis of $\textrm{Ker}(d)$ is precisely what we need to continue the induction. The induction can be started using the fact that $K_{-1}=0$. \end{remark}

Recall that we write $L(V)^k$ for the weight-$k$ component of $L(V)$, that is, the submodule generated by brackets of length $k$ in the elements of $V$, and recall also that weight defines a grading. Note that the differential $d$ preserves weight. The operations $\tau_k$ and $\sigma_k$ satisfy $$\wt(\tau_k(x))=p^k\wt(x),$$ $$\wt(\sigma_k(x))=p^k\wt(x).$$

We will use weight to produce a modified dimension function which makes precise the idea that `most' of the decomposition of $L(V)$ in Proposition \ref{acyclicHomology} consists of the summand $K$; the summand $HL(V)$ is comparatively small.

\begin{definition} Let $M$ be a $\Zpr$-module, together with a grading $\wt$, which we think of as a weight, such that each weight-component $M^i$ is free and finitely generated. Define $\dim^k(M) \in \mathbb{R}$ by setting $$\dim^k(M) = \sum_{i=1}^k \frac{\dim(M^i)}{i}.$$ \end{definition}

It follows immediately from the definition that $$\dim^k(A \oplus B) = \dim^k(A)+\dim^k(B).$$

We will be concerned with evaluating the functions $\dim^k$ on submodules of the free Lie algebra $L(V)$. We write $BM$ for the module $\textrm{Im}(d)$ of boundaries in a differential module $(M,d)$.

\begin{lemma} \label{weightInequalities} Let $V$ be an acyclic differential $\Zp$-vector space. For all $k \in \mathbb{N}$ we have: \begin{itemize}
    \item $\dim^k(HL(V)) < \frac{1}{p} \dim^k(L(V))$, and 
    \item $\dim^k(BL(V)) > \frac{p-1}{2p} \dim^k(L(V))$.
\end{itemize} \end{lemma}

\begin{proof} Decompose $L(V) \cong HL(V) \oplus K$ as in Proposition \ref{acyclicHomology}, and choose a basis $\{x_\alpha, y_\alpha, z_\beta, w_\beta\}$ of $K$ as in Remark \ref{basisChoice}, where $\alpha$ and $\beta$ run over indexing sets $\mathscr{I}$ and $\mathscr{J}$ respectively. The differential preserves weight, so by choosing such a basis in each weighted component separately, we may assume that the basis vectors are homogenous in weight. Let $S_k$ be the set of those $\alpha \in \mathscr{I}$ with $\wt(x_\alpha) \leq k$. Proposition \ref{acyclicHomology} gives that $$\dim^k(HL(V)) < \sum_{\alpha \in S_k} \sum_{j=1}^\infty \frac{1}{ \wt(\tau_j(x_\alpha))} + \frac{1}{\wt(\sigma_j(x_\alpha))} = \sum_{\alpha \in S_k} \sum_{j=1}^\infty \frac{1}{p^j \wt(x_\alpha)} + \frac{1}{p^j \wt(x_\alpha)}$$ $$= \sum_{\alpha \in S_k} \frac{2}{\wt(x_\alpha)} \sum_{j=1}^{\infty} \frac{1}{p^j} = \frac{1}{p-1} \sum_{\alpha \in S_k} \frac{2}{\wt(x_\alpha)}.$$ On the other hand, the contribution of the $x_\alpha$ and $y_\alpha$ to the dimension of $K$ gives that $$\dim^k(K) \geq \sum_{\alpha \in S_k} \frac{2}{\wt(x_\alpha)},$$ so $$\dim^k(K) > (p-1) \dim^k(HL(V)).$$ Since $L(V) \cong HL(V) \oplus K$, we have that $\dim^k(L(V)) = \dim^k(K) + \dim^k(HL(V))$, so $$\dim^k(L(V)) > p \dim^k(HL(V)),$$ which proves the first inequality. This also implies that $\dim^k(K) > \frac{p-1}{p} \dim^k(L(V))$, and since $K$ is acyclic, we must have $$\dim^k(BL(V)) \geq \frac{1}{2} \dim^k(K).$$ Combining these proves the second inequality and completes the proof. \end{proof}

All we will require for our application is the case when $V$ is the free $\Zpr$-module on two generators $x$ and $y$ satisfying $d(x)=y$. In this case we will write $L(x,dx)=L(V)$ and $L'(x,dx)=L'(V)$. Note that $L(x,dx) \otimes \Zps$ is the free Lie algebra on $V \otimes \Zps$, which is the free differential module over $\Zps$ on $x$ and $y$ with $dx=y$.

\begin{lemma} \label{ManyBoundaries} Let $V$ be a graded acyclic $\Zpr$-module, free and finitely generated in each dimension, of total dimension at least 2. Then the module of boundaries $BL(V)$ is $\Zpr$-hyperbolic. In particular, the module of boundaries $BL(x,dx)$ in the free differential Lie algebra $L(x,dx)$ is $\Zpr$-hyperbolic. \end{lemma}

\begin{proof} Since it has the correct universal property, $L(V) \otimes \Zp$ is the free Lie algebra over $\Zp$ on $V \otimes \Zp$. Thus, by Lemma \ref{reduceToField} applied to the differential $d$ it suffices to prove the $r=1$ case, for which we can use Proposition \ref{acyclicHomology}, in the guise of Lemma \ref{weightInequalities}.

By Lemma \ref{weightInequalities}, we know that $$\dim^k(BL(V)) > \frac{p-1}{2p} \dim^k(L(V)).$$ Thus, $$\sum_{i=1}^k \dim(BL(V)^i) \geq \sum_{i=1}^k \frac{\dim(BL(V)^i)}{i} > \frac{p-1}{2p} \sum_{i=1}^k \frac{\dim(L(V)^i)}{i} \geq \frac{p-1}{2p} \sum_{i=1}^k \frac{\dim(L(V)^i)}{k}.$$

Let $n$ be the maximum $i$ for which $V_i \neq 0$. The leftmost term is equal to $\dim(\bigoplus_{i=1}^k BL(V)^i)$, and $BL(V)^i \subset L(V)^i \subset L(V)_{ni}$, so we have $$\dim(\bigoplus_{j=1}^{nk} BL(V)_j) > \frac{p-1}{2pk} \sum_{i=1}^k \dim(L(V)^i) \geq \frac{p-1}{2pk} \dim(L(V)^k) = \frac{p-1}{2pk} W_\ell(k),$$ by Theorem \ref{gradedWittCount}, where we let $\ell = \dim(V)$, so $$\dim(\bigoplus_{i=1}^{k} BL(V)_i) > \frac{p-1}{2p \lfloor \frac{k}{n} \rfloor} W_\ell(\lfloor \frac{k}{n} \rfloor) \sim \frac{p-1}{2p \lfloor \frac{k}{n} \rfloor^2} \ell^{\lfloor \frac{k}{n} \rfloor}$$ by Lemma \ref{WittAsymptotics}. Now, $\ell$ is assumed greater than $1$, so
$$\frac{p-1}{2p \lfloor \frac{k}{n} \rfloor^2} \ell^{\lfloor \frac{k}{n} \rfloor} \geq \frac{p-1}{2p (\frac{k}{n})^2} \ell^{\frac{k}{n} -1},$$ so for any $\varepsilon > 0$, once $k$ is large enough we have $\dim(\bigoplus_{i=1}^{k} BL(V)_i) > (\ell^{\frac{1}{n}} - \varepsilon)^{k}$. That is, $\dim(\bigoplus_{i=1}^{k} BL(V)_i)$ grows faster than an exponential in any base smaller than $\ell^{\frac{1}{n}}$. In particular, if $\dim(V) = \ell \geq 2$, then $BL(V)$ is $\Zp$-hyperbolic, as required. \end{proof}

Since $\quot : L'(V) \longrightarrow L(V)$ is surjective and commutes with $d$, we immediately obtain the following corollary.

\begin{corollary} \label{mbCor} The submodule $\textrm{Im}(\quot \circ d)$ in the free differential Lie algebra $L(x,dx)$ is $\Zpr$-hyperbolic. \qed \end{corollary}

\section{Loop-homology of Moore spaces} \label{FoundationSection}

In this section we will study the question `what part of $H_*(\Omega P^{n+1}(p^r); \Zpr)$ can be shown to consist of Hurewicz images?' The answer is `the module of boundaries in a differential sub-Lie algebra isomorphic to $L(x,dx)$'. In Section \ref{LieSection} we have seen that such a module is $\Zpr$-hyperbolic. The hypotheses of Theorem \ref{HPrelim} are really conditions under which the image of this module under the map $(\Omega \mu)_*$ remains $\Zps$-hyperbolic, and we thus obtain a $\Zps$-hyperbolic submodule of the image of the Hurewicz map.

We follow the notation from Neisendorfer's book \cite{NeisendorferBook}. Let $p$ be a prime and let $s \leq r \in \mathbb{N}$. For a space $Y$, recall that the \textit{homotopy groups of $Y$ with coefficients in $\Zps$}, denoted $\pi_n(Y;\Zps)$ are the based homotopy sets $[P^n(p^s),Y]$, which are groups for $n \geq 3$. There are a number of useful operations relating the integral and mod-$p^s$ homotopy groups, which we introduce next.

Let $\underline{\beta}^s :S^{n-1} \longrightarrow P^n(p^s)$ be the inclusion from the cofibration sequence of Definition \ref{MooreSpace}. This defines a map of degree $-1$ $$\beta^s: \pi_n(Y;\Zps) \longrightarrow \pi_{n-1}(Y)$$ $$f \longmapsto f \circ \underline{\beta}^s.$$

Similarly, let $\underline{\rho}^s :P^n(p^s) \longrightarrow S^n$ be the pinch map, which is obtained by extending the cofibration sequence of Definition \ref{MooreSpace} to the right. Again, this defines a map of degree $0$ $$\rho^s: \pi_n(Y) \longrightarrow \pi_n(Y;\Zps)$$ $$f \longmapsto f \circ \underline{\rho}^s.$$

Lastly, let $\underline{\textrm{red}}^{r,s}: P^n(p^s) \longrightarrow P^n(p^r)$ be the map defined by the diagram of cofibrations \begin{center}
\begin{tabular}{c}
\xymatrix{
S^{n-1} \ar^{p^s}[r] \ar@{=}[d] & S^{n-1} \ar[r] \ar^{p^{r-s}}[d] & P^{n}(p^s) \ar^{\underline{\textrm{red}}^{r,s}}[d] \\
S^{n-1} \ar^{p^r}[r] & S^{n-1} \ar[r] & P^{n}(p^r),
}
\end{tabular}
\end{center} and let $$\textrm{red}^{r,s}: \pi_n(Y;\Zpr) \longrightarrow \pi_n(Y;\Zps)$$ $$f \longmapsto f \circ \underline{\textrm{red}}^{r,s}.$$

It follows from the definitions that $\beta^s$, $\rho^s$ and $\textrm{red}^{r,s}$ are all natural in $Y$.

We will now use these operations to produce elements $u$ and $v$ of $\pi_*(\Omega P^{n+1}(p^r); \Zps)$. The Hurewicz images of $v$ and $u$ will play the roles of the elements $x$ and $dx$ of Section \ref{LieSection}. Although these elements are easily described in terms of things we already have, we will give them new names for clarity.

Let $$v': P^n(p^s) \longrightarrow P^n(p^r)$$ be equal to $\underline{\textrm{red}}^{r,s}$.

Let $$u': P^{n-1}(p^s) \longrightarrow P^n(p^r)$$ be the composite $$P^{n-1}(p^s) \xrightarrow{\underline{\rho}^s} S^{n-1} \xrightarrow{\underline{\beta}^r} P^{n}(p^r).$$

Recall that for any space $X$ there is a natural map $\eta: X \longrightarrow \Omega \Sigma X$, which is the unit of the adjunction $\Sigma \dashv \Omega$ and sends $x \in X$ to the loop $\gamma_x = (t \longmapsto \langle t,x \rangle)$ on $\Sigma X$. Let $v = \eta \circ v' : P^n(p^s) \longrightarrow \Omega P^{n+1}(p^r)$, and let $u = \eta \circ u':  P^{n-1}(p^s) \longrightarrow \Omega P^{n+1}(p^r)$.

Now let $G$ be an $H$-group, and suppose that the prime $p$ is odd. As in the integral setting, the homotopy groups with coefficients $\pi_*(G; \Zps)$ carry a \textit{Samelson product}; a bilinear operation which resembles a Lie bracket \cite{CMNTorsion}. In particular, loop spaces are $H$-groups, so we have Samelson products in $\pi_*(\Omega X;\Zps)$ for any $X$.

\begin{lemma} \label{notQuiteGLA} Let $p$ be an odd prime. The map $$\pi_*(\Omega X;\Zps) \xrightarrow{\beta^s} \pi_*(\Omega X) \xrightarrow{\rho^s} \pi_*(\Omega X;\Zps)$$ is a differential (that is, $(\rho^s \circ \beta^s)^2=0$) of degree $-1$, which satisfies the Leibniz identity relative to Samelson products. \end{lemma}

\begin{proof} By \cite[Section 7]{CMNTorsion}, we have the Leibniz identity. To see that it is a differential, note that $\beta^s \circ \rho^s = 0$, so $(\rho^s \circ \beta^s)^2= \rho^s \circ (\beta^s \circ \rho^s) \circ \beta^s = 0$. \end{proof}

By construction of $u$ and $v$ we have $(\rho^s \circ \beta^s)(v)=p^{r-s}u$ in $\pi_*(\Omega P^{n+1}(p^r); \Zps)$. Let $L'(x,dx)$ and $L(x,dx)$ be as in Section \ref{LieSection}, where we let $\deg(x) = n$ and $\deg(y)=n-1$. Let $\langle x, dx \rangle$ be the free graded $\Zpr$-module of dimension 2 on basis $\{x, dx\}$, so that $L(x,dx) = L( \langle x,dx \rangle)$, and with this notation note that $L'(x,dx) \otimes \Zps = L'(\langle x,dx \rangle \otimes \Zps)$, the analogous construction over $\Zps$.

We define a map of $\Zps$-modules $\phi_\pi^{r,s} : \langle x, dx \rangle \otimes \Zps \longrightarrow \pi_*(\Omega P^{n+1}(p^r); \Zps)$ by sending $x \longmapsto v$ and $dx \longmapsto u$. Samelson products in $\pi_*(\Omega P^{n+1}(p^r); \Zps)$ are bilinear, so by the universal property of $L'(x,dx) \otimes \Zps$, $\phi_\pi^{r,s}$ extends to a map $$\PhiPi^{r,s} : L'(x,dx) \otimes \Zps \longrightarrow \pi_*(\Omega P^{n+1}(p^r); \Zps)$$ of graded (nonassociative) $\Zps$-algebras.

The following lemma relates $\PhiPi^{r,s}$ to $\PhiPi^{r,r}$.

\begin{lemma} \label{daFactor} If $s \leq r$ then $p^{r-s} \PhiPi^{r,s} \circ d = \rho^s \circ \beta^s \circ \PhiPi^{r,s}$. In particular, if $s=r$, then $\PhiPi^{r,s} = \PhiPi^{r,r}$ is a map of differential Lie algebras. \end{lemma}

\begin{proof} It suffices to show that the composites $p^{r-s} \PhiPi^{r,s} \circ d$ and $\rho^s \circ \beta^s \circ \PhiPi^{r,s}$ agree on brackets of length $k$ in $L'(x,dx) \otimes \Zps$ for each $k \in \mathbb{N}$. We will do this by induction.

In the case $k=1$, the restriction of $\PhiPi^{r,s}$ to brackets of length 1 is $\phi_\pi^{r,s}$. By construction of $u$ and $v$ we have $(\rho^s \circ \beta^s)(v)=p^{r-s}u$ in $\pi_*(\Omega P^{n+1}(p^r); \Zps)$, so $\phi_\pi^{r,s}$ satisfies $p^{r-s} \phi_\pi^{r,s} \circ d = \rho^s \circ \beta^s \circ \phi_\pi^{r,s}$, as required.

Now let $a \in L'(x,dx) \otimes \Zps$ be a bracket of length $k>1$. We have $a=[b,c]$ for brackets $b$, $c$ of lengths $i$ and $j$ respectively with $i+j=k$, $i<k$, $j<k$. Thus $$ \rho^s \circ \beta^s \circ \Phi_\pi^{r,s} (a) = \rho^s \circ \beta^s \circ \Phi_\pi^{r,s} ([b,c]) = \rho^s \circ \beta^s ([\Phi_\pi^{r,s} (b), \Phi_\pi^{r,s} (c)]) $$ $$ =  [\rho^s \circ \beta^s \circ \Phi_\pi^{r,s}( b), \Phi_\pi^{r,s} (c)] +(-1)^{\deg b }  [ \Phi_\pi^{r,s} (b), \rho^s \circ \beta^s \circ \Phi_\pi^{r,s} (c)], $$ where the last equality is by Lemma \ref{notQuiteGLA}. By induction we have $\rho^s \circ \beta^s \circ \Phi_\pi^{r,s}( b) = p^{r-s} \PhiPi^{r,s} \circ d(b)$ and $\rho^s \circ \beta^s \circ \Phi_\pi^{r,s}( c) = p^{r-s} \PhiPi^{r,s} \circ d(c)$, so the above is equal to $$ [p^{r-s} \PhiPi^{r,s} \circ d( b), \phi_\pi^{r,s} (c)] +(-1)^{\deg b }  [ \Phi_\pi^{r,s} (b), p^{r-s} \PhiPi^{r,s} \circ d (c)] = p^{r-s} \PhiPi^{r,s}( [  d( b), c] +(-1)^{\deg b }  [  b,  d (c)])$$ $$ = p^{r-s} \PhiPi^{r,s} \circ d( [  b, c]). $$

This completes the induction, and hence the proof. \end{proof} 

Lemma \ref{daFactor} identifies a factor of $p^{r-s}$. The next lemma makes precise the idea that this factor comes from the map $\beta^s$, rather than the map $\rho^s$, by relating each $\PhiPi^{r,s}$ to $\PhiPi^{r,r}$.

\begin{lemma} \label{PhiPiDiffs} The following diagram commutes: \begin{center}
\begin{tabular}{c}
\xymatrix{
L'(x,dx) \ar_{d}[dd] \ar^{\PhiPi^{r,r}}[r] & \pi_*(\Omega P^{n+1}(p^r); \Zpr) \ar^{\beta^s}[d] \\
& \pi_*(\Omega P^{n+1}(p^r)) \ar^{\rho^s}[d] \\
L'(x,dx) \otimes \Zps \ar^{\PhiPi^{r,s}}[r] & \pi_*(\Omega P^{n+1}(p^r); \Zps). \\
}
\end{tabular}
\end{center} In particular, $\textrm{Im}(\rho^s) \supset \textrm{Im}(\PhiPi^{r,s} \circ d)$. \end{lemma}

\begin{proof} By Lemma \ref{daFactor}, the top face of the following diagram commutes, and the bottom face commutes up to a factor of $p^{r-s}$, in the sense that $p^{r-s} \PhiPi^{r,s} \circ d = \rho^s \circ \beta^s \circ \PhiPi^{r,s}$:

\begin{center}
\begin{tabular}{c}
\xymatrix@C=0.1cm{
& \phantom{\Omega P^{n+1}(p^r)} & L'(x,dx) \ar^{\textrm{quotient}}[dd] \ar^{\PhiPi^{r,r}}[ddrr] \ar_{d}[ddll] & & \\
& & & & \\
L'(x,dx) \ar^{\textrm{quotient}}[dd] \ar^(.3){\PhiPi^{r,r}}[ddrr]  & & L'(x,dx) \otimes \Zps \ar'[dl]_{d}[ddll]  \ar'[dr]^{\PhiPi^{r,s}}[ddrr]& & \pi_*(\Omega P^{n+1}(p^r);\Zpr) \ar^{\textrm{red}^{r,s}}[dd] \ar_{\beta^r}[dl] \\
& \phantom{\ell} & & \pi_*(\Omega P^{n+1}(p^r)) \ar_{\rho^r}[dl] \ar@{=}[dd] &  \\
L'(x,dx) \otimes \Zps \ar^{\PhiPi^{r,s}}[ddrr] & & \pi_*(\Omega P^{n+1}(p^r); \Zpr) \ar^{\textrm{red}^{r,s}}[dd] & & \pi_*(\Omega P^{n+1}(p^r); \Zps) \ar_{\beta^s}[dl] \\
& & & \pi_*(\Omega P^{n+1}(p^r)) \ar_{\rho^s}[dl] & \\
& & \pi_*(\Omega P^{n+1}(p^r);\Zps). & & \\
}
\end{tabular}
\end{center}

Commutativity of the back left face is clear. We now check commutativity of the front left and back right faces, which are identical. Since the reduction map $\textrm{red}$ is a map of Lie algebras, both composites are maps of nonassociative algebras, and by the uniqueness part of the universal property of $L'(x,dx)$, it suffices to show that the restrictions to $\langle x, dx \rangle$ agree, and this is easily seen by direct calculation.

We now turn to the front right face. The square involving $\rho^s$ commutes, since the composite $$P^m(p^s) \xrightarrow{\underline{\textrm{red}}^{r,s}} P^m(p^r) \xrightarrow{\underline{\rho}^r} S^m $$ is equal to $\underline{\rho}^s: P^m(p^s) \longrightarrow S^m$. For the square involving $\beta^s$, we have that the composite $$S^{m-1} \xrightarrow{\underline{\beta}^s} P^m(p^s) \xrightarrow{\underline{\textrm{red}}^{r,s}} P^m(p^s)$$ is equal to $p^{r-s} \underline{\beta}^r : S^{m-1} \longrightarrow P^m(p^r)$.

Putting all of this together, we have that $$ \PhiPi^{r,s} \circ d \circ \textrm{quotient} = \textrm{red}^{r,s} \circ \PhiPi^{r,r} \circ d = \textrm{red}^{r,s} \circ \rho^r \circ \beta^r \circ \PhiPi^{r,r} = \rho^s \circ \beta^r \circ \PhiPi^{r,r},$$ as required. \end{proof}

Let $s \leq r$. The homology $\widetilde{H}_*(P^m(p^r);\Zps)$ is free over $\Zps$; in particular we have $$\widetilde{H}_i(P^m(p^r);\Zps) = \begin{cases} \Zps & i = m,m-1, \\
0  & \textrm{ otherwise.}
\end{cases} $$

Write $e_m$ for a choice of generator of $H_m(P^m(p^r);\Zps)$, and $s_{m-1} = \beta(e_m)$, where $\beta$ is the homology Bockstein. The group $H_{m-1}(P^m(p^r);\Zps)$ is generated by $s_{m-1}$.

The Pontrjagin product makes $\widetilde{H}_*(\Omega  P^{n+1}(p^r);\Zps)$ into a $\Zps$-algebra. Any graded associative algebra carries a Lie bracket, defined by setting $[x,y]=xy-(-1)^{\deg(x)\deg(y)}yx$, and this is what will be meant by `the bracket on $H_*(\Omega  P^{n+1}(p^r);\Zps)$'.

Recall that an element of $\pi_m(Y; \Zpr)$ is a homotopy class of maps $P^m(p^r) \longrightarrow Y$. Let $h: \pi_*(Y; \Zps) \longrightarrow H_*(Y;\Zps)$ be the \emph{Hurewicz map}, which sends $f \in \pi_*(Y; \Zps)$ to $f_*(e_m) \in H_*(Y;\Zps)$. By \cite[Proposition 6.4]{CMNTorsion}, the generators $e_m$ may be chosen so that $h$ carries Samelson products to commutators; that is, so that $h([f,g]) = [h(f),h(g)] \in H_*(\Omega P^{n+1}(p^r); \Zps)$.

Thus, the composition $h \circ \PhiPi^{r,s}$ respects brackets, and the codomain, $H_*(\Omega P^{n+1}(p^r); \Zps)$ carries a (genuine) Lie algebra structure. We therefore obtain a factorization of $h \circ \PhiPi^{r,s}$ through $\quot$ to give a map of Lie algebras $\PhiH^{r,s}$ which satisfies the following lemma:

\begin{lemma} \label{hurewiczTriangle} The following diagram commutes: \begin{center}
\begin{tabular}{c}
\xymatrix{
L'(x,dx) \otimes \Zps \ar^{\PhiPi^{r,s}}[r] \ar^{\quot}[d]  & \pi_*(\Omega P^{n+1}(p^r); \Zps) \ar^{h}[d] \\
L(x,dx) \otimes \Zps \ar^{\PhiH^{r,s}}[r] & H_*(\Omega P^{n+1}(p^r); \Zps).
}
\end{tabular}
\end{center} \qed \end{lemma}

\subsection{Tensor algebras and the Bott-Samelson Theorem} \label{TBSSubsection}

The purpose of this section is to introduce some notation for dealing with tensor algebras, and to recall the Bott-Samelson Theorem (Theorem \ref{BS}). We define the \textit{tensor algebra} on a graded $R$-module $V$ to be $T(V) = \bigoplus_{k=1}^\infty V^{\otimes k}$, where $V^{\otimes k}$ is the tensor product of $k$ copies of $V$. In particular, this definition is `reduced' since we do not insert a copy of $R$ in degree 0. The multiplication is given by concatenation of tensors, and makes $T(V)$ into the free graded associative algebra on $V$. Let $A$ be an algebra and let $\varphi : V \longrightarrow A$ be a homomorphism. We write $\widetilde{\varphi} : T(V) \longrightarrow A$ for the map of algebras induced by $\varphi$. Let $$\iota_i : V^{\otimes i} \longrightarrow T(V)$$ be the inclusion, and let $$\zeta_i : T(V) \longrightarrow V^{\otimes i}$$ be the projection.

Bott and Samelson first proved their theorem in \cite{BottSamelson}; we give the formulation from Selick's book \cite{Selick}.

\begin{theorem}[Bott-Samelson] \label{BS} Let $R$ be a PID, and let $X$ be a connected space with $\widetilde{H}_*(X;R)$ free over $R$. Then $\widetilde{H}_*(\Omega \Sigma X;R) \cong T(\widetilde{H}_*(X;R))$ and $\eta: X \longrightarrow \Omega \Sigma X$ induces the canonical map $\widetilde{H}_*(X;R) \longrightarrow T(\widetilde{H}_*(X;R))$. \end{theorem}

The Bott-Samelson Theorem immediately allows us to find a free Lie algebra in the loop-homology of a Moore space.

\begin{lemma} \label{injection1} The map $\PhiH^{r,s} : L(x,dx) \otimes \Zps \longrightarrow H_*(\Omega P^{n+1}(p^r); \Zps)$ is an injection. \end{lemma}

\begin{proof} Since $r \geq s$, the module $H_*( P^{n}(p^r); \Zps)$ is free over $\Zps$. By the Bott-Samelson Theorem \ref{BS}, $\widetilde{H}_*(\Omega P^{n+1}(p^r); \Zps) \cong T(x,dx) \otimes \Zps$, and this isomorphism identifies $\PhiH^{r,s}$ with the natural map $L(x,dx) \otimes \Zps \longrightarrow T(x,dx) \otimes \Zps$. But this latter map is an injection by Proposition 2.9 and Corollary 2.7 of \cite{CMNTorsion}. \end{proof}

We have the following corollary, which will be the main ingredient in the proof of Theorem \ref{HPrelim}.

\begin{corollary} \label{combining} Let $Y$ be a simply connected $CW$-complex, let $p$ be an odd prime, and let $r \in \mathbb{N}$. Let $\mu: P^{n+1}(p^r) \longrightarrow Y$ be a continuous map. If the induced map $$(\Omega \mu)_* : H_*(\Omega P^{n+1}(p^r);\Zps) \longrightarrow H_*(\Omega Y;\Zps)$$ is an injection, then the module $\textrm{Im}((\Omega \mu)_* \circ \PhiH^{r,s} \circ \quot \circ d)$ is $\Zps$-hyperbolic. \qed \end{corollary}

\begin{proof} By Lemma \ref{injection1}, $\PhiH^{r,s}$ is an injection, and by Corollary \ref{mbCor}, the module $\textrm{Im}(\quot \circ d)$ is $\Zps$-hyperbolic. It follows that $(\Omega \mu)_* \circ \PhiH^{r,s} (\textrm{Im}(\quot \circ d)) = \textrm{Im}((\Omega \mu)_* \circ \PhiH^{r,s} \circ \quot \circ  d)$ is also $\Zps$-hyperbolic. \end{proof}

\section{The suspension case} \label{SuspensionSection}

The purpose of this section is to show that Theorem \ref{HPrelim} implies Theorem \ref{HCriterion}. This will be accomplished by means of Proposition \ref{injection2}, whose proof is the goal of this section. The main point is that even if $\widetilde{H}_*(X; \Zps)$ is not free over $\Zps$, the canonical map of the Bott-Samelson Theorem (Theorem \ref{BS}) is still an injection. That is, the homology $\widetilde{H}_*(\Omega \Sigma X ; \Zps)$ always contains the tensor algebra on $\widetilde{H}_*(X; \Zps)$, but if $\widetilde{H}_*(X; \Zps)$ is not free then it will contain other things too.

In Subsection \ref{StructureSubs}, we recall the James splitting $\Sigma \Omega \Sigma X \simeq \bigvee_{k=1}^\infty \Sigma X^{\wedge k}.$ This gives us Proposition \ref{BSGen}, which describes the structure of the Pontrjagin algebra $\widetilde{H}_*(\Omega \Sigma X ; \Zps)$, in particular identifying the tensor algebra $T(\widetilde{H}_*(X; \Zps))$ as a subalgebra. Subsection \ref{evSubs} proves Lemma \ref{evproj}, which describes the effect of the evaluation map on $H_*(\Omega \Sigma X ; \Zps)$. Subsection \ref{FinnickySubs} draws these ingredients together to prove Proposition \ref{injection2}.

Let $\sigma : \widetilde{H}_*(Y) \longrightarrow \widetilde{H}_{*+1}(\Sigma Y)$ denote the suspension isomorphism. For a space $X$, let $X^k$ denote the product of $k$ copies of $X$, and let $X^{\wedge k}$ denote the smash product. Let $\sim$ be the relation on $X^k$ defined by $$(x_1, \dots ,x_{i-1}, * , x_{i+1}, x_{i+2}, \dots x_{k}) \sim (x_1, \dots , x_{i-1} , x_{i+1}, * , x_{i+2} , \dots x_{k}).$$
Let $J_k(X)$ be the space $\faktor{X^k}{\sim}$. There is a natural inclusion $$J_k(X) \longrightarrow J_{k+1}(X)$$ $$(x_1, \dots , x_k) \mapsto (x_1, \dots, x_k, *).$$

The \emph{James construction} $JX$ is defined to be the colimit of the diagram consisting of the spaces $J_k(X)$ and the above inclusions. Notice that $JX$ carries a product given by concatenation, which makes it into the free topological monoid on $X$, and that a topological monoid is in particular an $H$-space.

The adjunction isomorphism $[\Sigma X, Y] \cong [X, \Omega Y]$ will be written in both directions as $f \longmapsto \overline{f}$. Recall that $\eta$ denotes the unit of the adjunction, which is the map $X \longrightarrow \Omega \Sigma X$ sending $x \in X$ to $(t \mapsto \langle t,x \rangle) \in \Omega \Sigma X$. We will write $\ev$ for the \textit{evaluation map}; the counit $\Sigma \Omega Y \longrightarrow Y$, which sends $\langle t, \gamma \rangle \in \Sigma \Omega Y$ to $\gamma(t) \in Y$.

\subsection{The tensor algebra inside $H_*(\Omega \Sigma X)$} \label{StructureSubs}

In this section we will generalise the Bott-Samelson theorem to suit our purpose. Specifically, the map $\eta : X \longrightarrow \Omega \Sigma X$ induces a map $\eta_* : \widetilde{H}_*(X) \longrightarrow \widetilde{H}_*(\Omega \Sigma X)$ on homology. By the universal property of the tensor algebra, $\eta_*$ extends to a map of algebras $$ \widetilde{\eta_*} : T(\widetilde{H}_*(X)) \longrightarrow \widetilde{H}_*(\Omega \Sigma X).$$

The Bott-Samelson Theorem (Theorem \ref{BS}) says that if the homology $H_*(X;\Zps)$ is free then $\widetilde{\eta_*}$ is an isomorphism. We will show that even if $H_*(X;\Zps)$ is not free, the map $\widetilde{\eta_*}$ is still an injection. This is by no means new, but follows reasonably easily from better-known results, so we shall derive it in this way. In this section homology is taken with $\Zps$-coefficients (unless otherwise stated).

\begin{lemma} \label{crossCoker} The cross product map $\widetilde{H}_*(X)^{\otimes k} \xrightarrow{\times} \widetilde{H}_*(X^{\wedge k})$ is injective, split (although not naturally) and its cokernel $C$ satisfies $p^{s-1} C = 0$. \end{lemma}

\begin{proof} For spaces $A$ and $B$ The K\"unneth Theorem gives an exact sequence $$0 \longrightarrow H_*(A) \otimes H_*(B) \xrightarrow{\times} H_*(A \times B) \longrightarrow \textrm{Tor}(H_*(A), H_{*-1}(B)) \longrightarrow 0,$$ where the $\textrm{Tor}$ is taken over $\Zps$, and this sequence is (unnaturally) split. By Lemma \ref{TorAnnihilation} we have $p^{s-1}\textrm{Tor}(H_*(A), H_{*-1}(B)) = 0$.

Let $a_0: pt \longrightarrow A$ denote the inclusion of the basepoint of $A$ and let $b_0$ denote the inclusion of the basepoint of $B$. Let $j: H_*(A) \oplus H_*(B) \longrightarrow H_*(A) \otimes H_*(B)$ be the composite $$H_*(A) \oplus H_*(B) \xrightarrow{\cong} H_*(A) \otimes H_*(pt) \oplus H_*(pt) \otimes H_*(B) \xrightarrow{(id_A)_* \otimes (b_0)_* \oplus (a_0)_* \otimes (id_B)_*} H_*(A) \otimes H_*(B).$$ To relate the reduced and unreduced situations we have the following diagram (which we take to define the reduced cross product) where $i$, $i_1$ and $i_2$ are the inclusions and $p$ is the quotient.  \begin{center}
\begin{tabular}{c}
\xymatrix{
0 \ar[r] & H_*(A) \oplus H_*(B) \ar@{^{(}->}^j[d] \ar^(.55){(i_1)_* \oplus (i_2)_*}[r] & H_*(A \vee B) \ar@{^{(}->}^{i_*}[d] \ar[r] & 0 \ar[d] \ar[r] & 0 \\
0 \ar[r] & H_*(A) \otimes H_*(B) \ar^(.55){\times}[r] \ar[d] & \ar[r] H_*(A \times B) \ar^{p_*}[d] & \textrm{Tor}(H_*(A), H_{*-1}(B)) \ar@{=}[d] \ar[r] & 0 \\
0 \ar[r] & \widetilde{H}_*(A) \otimes \widetilde{H}_*(B) \ar^(.55){\times}[r] & \widetilde{H}_*(A \wedge B) \ar[r] & \textrm{Tor}(H_*(A), H_{*-1}(B)) \ar[r] & 0.
}
\end{tabular}
\end{center} The bottom row is obtained from the other two by taking cokernels, so is automatically exact, and it therefore suffices to check that the top two squares commute. The top right square commutes because the map $(i_1)_* \oplus (i_2)_*$ is an isomorphism, so the composite of $i_*$ with the map into the $\textrm{Tor}$ term factors through two terms of an exact sequence, hence is zero, as required.

We now check that the top left square commutes. It suffices to check commutativity on each summand of the domain individually. We will do so for $H_*(A)$; the case of $H_*(B)$ is analogous. Identifying $H_*(A)$ with $H_*(A) \otimes H_*(pt)$, the restriction of $j$ becomes $(id_A)_* \otimes (b_0)_*$. The composite with the cross product is written $(id_A)_* \times (b_0)_*$, and by bilinearity of cross product this is the same as $(id_A \times b_0)_*$, where now the product is taken in spaces. But under the identification $A \cong A \times \{pt\}$, this is just the inclusion $A \longrightarrow A \times B$, which is the map obtained by going the other way round the square, as required.

To finish, we note that since the middle row is split, the bottom row is also split. \end{proof}

The understanding of the cross product from Lemma \ref{crossCoker} allows us to understand part of the homology of $JX$, by constructing a map $\varphi$ as in the following lemma.

\begin{lemma} \label{JamesHomology} The maps $$\widetilde{H}_*(X)^{\otimes k} \longrightarrow H_*(X)^{\otimes k} \xrightarrow{\times} H_*(X^k) \longrightarrow H_*(J_k(X)) \longrightarrow H_*(J(X))$$ define an injection of algebras $T(\widetilde{H}_*(X)) \xrightarrow{\JTens} \widetilde{H}_*(J(X)).$ Furthermore, $\textrm{Im}(\JTens)$ is a direct summand, and we may write $\widetilde{H}_*(J(X)) \cong T(\widetilde{H}_*(X)) \oplus C$ such that the complementary module $C$ satisfies $p^{s-1} C = 0$. \end{lemma}

\begin{proof} We use a modified version of the argument in \cite[Proposition 3C.8]{Hatcher}. First, $\JTens$ is a ring homomorphism, because the product in $J(X)$ descends from the natural map $X^i \times X^j \longrightarrow X^{i+j}$. To see that we have an injection, we consider the following diagram, where we follow Hatcher's notation and set $T_k(M) = \bigoplus_{i=1}^k M^{\otimes i}$: \begin{center}
\begin{tabular}{c}
\xymatrix{
0 \ar[r] & T_{k-1}(\widetilde{H}_*(X)) \ar^{\JTens}[d] \ar[r] & T_{k}(\widetilde{H}_*(X)) \ar^{\JTens}[d] \ar[r] & (\widetilde{H}_*(X))^{\otimes k} \ar^{\times}[d] \ar[r] & 0 \\
0 \ar[r] & \widetilde{H}_*(J_{k-1}(X)) \ar[r] & \widetilde{H}_*(J_k(X)) \ar[r] & \widetilde{H}_*(X^{\wedge k}) \ar[r] & 0.
}
\end{tabular}
\end{center}

Commutativity of the diagram follows from the definition of $\JTens$. Exactness of the top row is clear. The bottom row is obtained from the long exact sequence of the pair $(J_k(X), J_{k-1}(X))$, applying excision to pass to the quotient $J_k(X)/J_{k-1}(X) \simeq X^{\wedge k}$. This sequence is split because the quotient $X^k \longrightarrow X^{\wedge k}$ factors through the map $J_k(X) \longrightarrow X^{\wedge k}$, and the former map is split after suspending. Thus we get that $\widetilde{H}_*(J_k(X)) \cong \widetilde{H}_*(J_{k-1}(X)) \oplus \widetilde{H}_*(X^{\wedge k})$. Lemma \ref{crossCoker} tells us that $\widetilde{H}_*(X^{\wedge k}) \cong (\widetilde{H}_*(X))^{\otimes k} \oplus C$ with $p^{s-1} C = 0$, so the result follows immediately by inducting over $k$. \end{proof}

Our next job is to translate this understanding of $JX$ into an understanding of $\Omega \Sigma X$. It is well-known that the two are homotopy equivalent, but we wish to be precise about the maps. For a based space $Y$, let $\Omega' Y$ denote the space of loops of any length in $Y$, so that $\Omega Y$ is the subspace of $\Omega' Y$ consisting of loops of length 1. We will write $\gamma_1 \# \gamma_2$ for the concatenation of loops $\gamma_1$ and $\gamma_2$. For $\gamma \in \Omega' Y$ and $\ell \in \mathbb{R}_{>0}$, let $\gamma^{\ell}$ denote the linear reparameterization of $\gamma$ which has length $\ell$. Note that $\gamma \longmapsto \gamma^1$ is a continuous map $\Omega' Y \longrightarrow \Omega Y$, which is a retraction for the inclusion $\Omega Y \subset \Omega' Y$. For $x \in X$, let $\gamma_x \in \Omega \Sigma X$ be the loop defined by $\gamma_x(t) = \langle t, x \rangle$, which is equal to $\eta(x)$.

Now let $X$ be a connected $CW$-complex, which we take without loss of generality to have a single 0-cell, which is the basepoint. Let $d: X \longrightarrow [0,1]$ be any continuous map such that $d^{-1}(0)=\{*\}$.  Define a map $$\lambda : J(X) \longrightarrow \Omega \Sigma X$$ $$(x_1, \dots x_k) \longmapsto (\gamma_{x_1}^{d(x_1)} \# \gamma_{x_2}^{d(x_2)} \# \dots \# \gamma_{x_k}^{d(x_k)})^1.$$

The reparameterization is necessary so that $\lambda$ is well-defined when some $x_i = *$.

Hatcher proves the following as \cite[Theorem 4J.1]{Hatcher}.

\begin{lemma} \label{LAMBDA} The map $\lambda$ is a weak homotopy equivalence for any connected $CW$-complex $X$. Furthermore, $\lambda$ is an $H$-map, so it induces a map of algebras on homology. \qed \end{lemma}

The following lemma is immediate from the definition of $\lambda$.

\begin{lemma} \label{describeLambda} The composite $$X^k \rightarrow J_k(X) \rightarrow J(X) \xrightarrow{\lambda} \Omega \Sigma X$$ is homotopy equivalent to $m \circ \eta^k$, where $m$ is any choice of $k$-fold loop multiplication on $\Omega \Sigma X$. \qed \end{lemma}

Recall from Subsection \ref{TBSSubsection} that $\iota_k : V^{\otimes k} \longrightarrow T(V)$ is the inclusion.

We are now ready to prove the main result of this subsection, which is what we will use later.

\begin{proposition} \label{BSGen} The map $$\widetilde{\eta_*} : T(\widetilde{H}_*(X)) \longrightarrow H_*(\Omega \Sigma X)$$ is an injection onto a summand, each restriction $\widetilde{\eta_*} \circ \iota_k$ is equal to $$ \widetilde{H}_*(X)^{\otimes k} \xrightarrow{\times} \widetilde{H}_*(X^k) \xrightarrow{(\eta^k)_*} \widetilde{H}_*((\Omega \Sigma X)^k) \xrightarrow{m_*} \widetilde{H}_*(\Omega \Sigma X),$$ and we may write $$H_*(\Omega \Sigma X) \cong T(\widetilde{H}_*(X)) \oplus C$$ such that the complementary module $C$ satisfies $p^{s-1} C=0$.
\end{proposition}

\begin{proof} By Lemmas \ref{JamesHomology}, \ref{LAMBDA} and \ref{describeLambda}, it suffices to show that $\lambda_* \circ \JTens = \widetilde{\eta_*}$. Since both maps are algebra maps, by the universal property of the tensor algebra it further suffices to show that the composite $$\widetilde{H}_*(X) \xrightarrow{\iota_1} T(\widetilde{H}_*(X)) \xrightarrow{\JTens} \widetilde{H}_*(JX) \xrightarrow{\lambda_*} \widetilde{H}_*(\Omega \Sigma X) $$ is equal to $\eta_*$. 

To see this, first note that the composite $\widetilde{H}_*(X) \xrightarrow{\iota_1} T(\widetilde{H}_*(X)) \xrightarrow{\JTens} \widetilde{H}_*(JX)$ is equal to the map induced by the inclusion $X \longrightarrow J_1(X) \subset J(X)$ which carries $x \in X$ to the equivalence class of $x$ in $J(X)$. By definition of $\lambda$ we then have $\lambda(x) = \gamma_x$, which by definition is $\eta(x)$, as required. \end{proof}

\subsection{The effect of the evaluation map} \label{evSubs}

The goal of this section is to prove Lemma \ref{evproj}, which says that up to suspension isomorphisms, the evaluation map $\ev: \Sigma \Omega \Sigma X \longrightarrow \Sigma X$ induces the projection onto the tensors of length 1. Our strategy is to first prove Lemma \ref{niceSquare}, the point of which is that when one evaluates a concatenation of $k$ loops at some time $t$, the result only depends on one of the loops - this is the $i$ appearing in the proof of Lemma \ref{niceSquare}. We will then see that this, together with simple formal properties of the cross product, is enough to prove Lemma \ref{evproj}.

In this section, for a co-$H$-space $Y$, $c: Y \longrightarrow Y \vee Y$ denotes the comultiplication, and for a product $\prod_{i=1}^k X_i$, the map $\pi_i$ is the projection onto the $i$-th factor. In the next lemma we take the iterated comultiplication $c$ and the iterated multiplication $m$ to be parameterized so as to spend equal time on each component - this does not change anything up to homotopy.

\begin{lemma} \label{niceSquare} The following diagram commutes. \begin{center}
\begin{tabular}{c}
\xymatrix{
\Sigma X^k \ar^{c}[d] \ar^{\Sigma \eta^k}[r] & \Sigma (\Omega \Sigma X)^k \ar^{\Sigma m}[r] & \Sigma \Omega \Sigma X \ar^{\ev}[d]\\
(\Sigma X^k)^{\vee k} \ar^{\bigvee_{i=1}^k \Sigma \pi_i}[d] & & \Sigma X \\
(\Sigma X)^{\vee k} \ar^{\textrm{fold}}[r] & \Sigma X \ar^{\Sigma \eta}[r] & \Sigma \Omega \Sigma X. \ar_{\ev}[u] \\
}
\end{tabular}
\end{center} \end{lemma}

\begin{proof} We will evaluate both composites. A point of $\Sigma X^k$ may be written in suspension coordinates as $\langle t, x_1, x_2, \dots, x_k \rangle$, for $t \in I$ and $x_i \in X$. There exists some integer $i$ with $1 \leq i \leq k$ so that $\frac{i-1}{k} \leq t \leq \frac{i}{k}$.

For the top right composite, $$\ev \circ \Sigma m \circ \Sigma \eta^k \langle t, x_1, \dots, x_k \rangle = \ev \langle t, m(\gamma_{x_1}, \dots , \gamma_{x_k}) \rangle = (\gamma_{x_1} \# \dots \# \gamma_{x_k})(t)= \gamma_{x_i}(kt - (i-1)).$$

For the bottom left composite, we first introduce some notation. For a point $y$ of a space $Y$, we write $(y)_i$ for the image of $y$ under the inclusion of the $i$-th wedge summand in $Y \longrightarrow Y^{\vee k}$. With this notation, taking $Y = \Sigma X^k$, we have $c \langle t, x_1, \dots, x_k \rangle = (\langle kt-(i-1), x_1, \dots, x_k \rangle)_i $. Therefore, $$\ev \circ \Sigma \eta \circ \textrm{fold} \circ (\bigvee_{i=1}^k \Sigma \pi_i) \circ c \langle t, x_1, \dots, x_k \rangle = \ev \circ \Sigma \eta \circ \textrm{fold}(\langle kt-(i-1), x_i \rangle)_i$$ $$=\ev \circ \Sigma \eta \langle kt-(i-1), x_i \rangle = \gamma_{x_i}(kt-(i-1)),$$ as required. \end{proof}

\begin{lemma} \label{crossWorks} Let $X$ be a space. The composite $$\widetilde{H}_*(X)^{\otimes k} \xrightarrow{\times} \widetilde{H}_*(X^k) \xrightarrow{(\pi_i)_*} \widetilde{H}_*(X)$$ of the cross product with any projection is trivial for $k \geq 2$. \end{lemma}

\begin{proof} Up to homeomorphism, $X$ may be regarded as the space $\prod_{j=1}^k Y_j$, where $Y_j = *$ for $j \neq i$ and $Y_i = X$. Under this identification, $\pi_i$ is identified with the map $\prod_{j=1}^k f_j : X^k \longrightarrow \prod_{j=1}^k Y_j$, where $f_j$ is the identity on $X$ when $j=i$, and is the trivial map otherwise.

The composite of maps $(\prod_{j=1}^k f_j)_* \circ \times$ is the cross product of homomorphisms $(f_1)_* \times (f_2)_* \times \dots \times (f_k)_*$. Cross product of homomorphisms is $k$-multilinear, and since $k \geq 2$ there is at least one $j$ with $f_j$ equal to the constant map, hence $(f_j)_*=0$. This means that $(\prod_{j=1}^k f_j)_* \circ \times$ is trivial for $k \geq 2$, as required. \end{proof}

\begin{corollary} \label{crossPlus} Let $X$ be a space. The composite $$\widetilde{H}_*(X)^{\otimes k} \xrightarrow{\times} \widetilde{H}_*(X^k) \xrightarrow{\sigma} \widetilde{H}_*(\Sigma X^k) \xrightarrow{c_*} \widetilde{H}_*((\Sigma X^k)^{\vee k}) \xrightarrow{(\bigvee_{i=1}^k \Sigma \pi_i)_*} \widetilde{H}_*((\Sigma X)^{\vee k})$$ is trivial for $k \geq 2$. \end{corollary}

\begin{proof} For a space $Y$, let $p_i : Y^{\vee k} \longrightarrow Y$ be the projection onto the $i$-th wedge summand. The comultiplication $c$ satisfies $p_i \circ c \simeq id_{\Sigma X^k}$ for each $i$, so on homology we have $$c_* : \widetilde{H}_*(\Sigma X^k) \longrightarrow \widetilde{H}_*((\Sigma X^k)^{\vee k}) \cong \bigoplus^{k}_{i=1} \widetilde{H}_*(\Sigma X^k)$$ $$x \longmapsto (x, x, \dots, x).$$

That is, $c_*$ may be identified with the diagonal map $\Delta : \widetilde{H}_*(\Sigma X^k) \longrightarrow \bigoplus^{k}_{i=1} \widetilde{H}_*(\Sigma X^k)$.

Thus, $$(\bigvee_{i=1}^k \Sigma \pi_i)_* \circ c_* \circ \sigma \circ \times (x_1 \otimes \dots \otimes x_k) = (\bigvee_{i=1}^k \Sigma \pi_i)_* \circ c_*( \sigma (x_1 \times \dots \times x_k))$$ $$ = \bigoplus_{i=1}^k (\Sigma \pi_i)_* \circ \Delta (\sigma (x_1 \times \dots \times x_k))=0,$$ since by Lemma \ref{crossWorks} we have $$(\Sigma \pi_i)_* (\sigma (x_1 \times \dots \times x_k)) = \sigma \circ (\pi_i)_* (x_1 \times \dots \times x_k)=0.$$

This completes the proof. \end{proof}

\begin{lemma} \label{evproj} The composite $$T(\widetilde{H}_*(X)) \xrightarrow{\widetilde{\eta_*}} \widetilde{H}_*(\Omega \Sigma X) \xrightarrow{\sigma} \widetilde{H}_*(\Sigma \Omega \Sigma X) \xrightarrow{\ev_*} \widetilde{H}_*(\Sigma X) \xrightarrow{\sigma^{-1}} \widetilde{H}_*(X)$$ is equal to the projection $\zeta_1$.
\end{lemma}

\begin{proof} Write $\Gamma$ for the above composite. We must show that $\Gamma \circ \iota_k$ is the identity map on $H_*(X)$ when $k=1$, and is $0$ otherwise.

For the $k=1$ statement, note that $\widetilde{\eta_*} \circ \iota_1 = \eta_*$ (this is the definition of $\widetilde{\eta_*}$). We may therefore write $$ \Gamma \circ \iota_1 = \sigma^{-1} \circ \textrm{ev}_* \circ \sigma \circ \widetilde{\eta_*} \circ \iota_1 = \sigma^{-1} \circ \textrm{ev}_* \circ \sigma \circ \eta_* = \sigma^{-1} \circ \textrm{ev}_* \circ (\Sigma \eta)_* \circ \sigma,$$ and by the triangle identities for the adjunction $\Sigma \dashv \Omega$ we have a commuting diagram \begin{center}
\begin{tabular}{c}
\xymatrix{
\Sigma X \ar^{id_{\Sigma X}}[dr] \ar^{\Sigma \eta}[r] & \Sigma \Omega \Sigma X \ar^{\textrm{ev}}[d] \\
& \Sigma X.
}
\end{tabular}
\end{center}

Thus, $\Gamma \circ \iota_1 = \sigma^{-1} \circ \sigma = id_{H_*(X)}$, as we required.

Now let $k>1$. Juxtaposing the diagram of Lemma \ref{niceSquare} (after taking homology) with the result of Corollary \ref{crossPlus} gives a commuting diagram \begin{center}
\begin{tabular}{c}
\xymatrix{
\widetilde{H}_*(X)^{\otimes k} \ar^{\times}[r] \ar^{0}[ddrr] & \widetilde{H}_*(X^k) \ar^{\sigma}[r] & \widetilde{H}_*(\Sigma X^k) \ar^{c_*}[d] \ar^{(\Sigma \eta^k)_*}[r] & \widetilde{H}_*(\Sigma (\Omega \Sigma X)^k) \ar^{(\Sigma m)_*}[r] & \widetilde{H}_*(\Sigma \Omega \Sigma X) \ar^{\ev_*}[d] \\
& & \widetilde{H}_*((\Sigma X^k)^{\vee k}) \ar^{(\bigvee_{i=1}^k \Sigma \pi_i)_*}[d] & & \widetilde{H}_*(\Sigma X) \\
& & \widetilde{H}_*((\Sigma X)^{\vee k}) \ar^{\textrm{fold}_*}[r] & \widetilde{H}_*(\Sigma X) \ar^{(\Sigma \eta)_*}[r] & \widetilde{H}_*(\Sigma \Omega \Sigma X). \ar_{\ev_*}[u] \\
}
\end{tabular}
\end{center}

The description of $\widetilde{\eta_*} \circ \iota_k$ of Proposition \ref{BSGen} implies that the top-right route round the diagram is equal to $\sigma \circ \Gamma \circ \iota_k$. The diagram shows that this factors through the zero map, so $\sigma \circ \Gamma \circ \iota_k = 0$, and since $\sigma$ is an isomorphism, this implies that $\Gamma \circ \iota_k$ is itself zero, which completes the $k>1$ case and hence the proof. \end{proof}

\subsection{Loops on homology injections} \label{FinnickySubs}

The goal of this section is to prove Proposition \ref{injection2}. We first prove two lemmas. Recall that $\iota_i : V^{\otimes i} \longrightarrow T(V)$ is the inclusion, and that $\zeta_i : T(V) \longrightarrow V^{\otimes i}$ is the projection. Similarly, let $\iota_{\leq k}$ and $\zeta_{\leq k}$ be the inclusion and projection associated to the submodule $\bigoplus_{i=1}^k V^{\otimes i}$ of $T(V)$.

\begin{lemma} \label{leadingTrims} Let $a_1, a_2, \dots a_k$ be elements of a tensor algebra $T(V)$. We have that $$\zeta_i (a_1 \otimes \dots \otimes a_k) = \begin{cases} \zeta_1(a_1) \otimes \dots \otimes \zeta_1(a_k) & i=k \\
0 & i<k.
\end{cases} $$ \end{lemma}

\begin{proof} By definition, $T(V) \cong \bigoplus_{i=1}^{\infty} V^{\otimes i}$, and the maps $\zeta_i$ are precisely the projections onto these summands. Further, the multiplication in $T(V)$ restricts to maps $V^{\otimes i} \otimes V^{\otimes j} \longrightarrow V^{\otimes (i+j)}$, which is to say that it is additive in weight. This gives the formula $\zeta_k(a \otimes b) = \sum_{i=1}^{k-1} \zeta_i(a) \otimes \zeta_{k-i}(b)$, which we will use to induct.

When $k=1$ the result is automatic. Assuming the result for $k-1$, we have $$\zeta_j(a_1 \otimes \dots \otimes a_k) = \sum_{i=1}^{j-1} \zeta_i(a_1 \otimes \dots \otimes a_{k-1}) \otimes \zeta_{j-i}(a_k).$$ 

By induction $\zeta_i(a_1 \otimes \dots \otimes a_{k-1}) = 0$ for $i<k-1$, so the above is $0$ when $j<k$ and when $j=k$ it becomes $$\zeta_{k-1}(a_1 \otimes \dots \otimes a_{k-1}) \otimes \zeta_1(a_k) = \zeta_1(a_1) \otimes \dots \otimes \zeta_1(a_{k-1}) \otimes \zeta_1(a_k),$$ by induction, as required. \end{proof}

The following lemma does not depend on the algebra structure in the tensor algebras; only on the fact that tensor algebras are graded by weight. Nonetheless, we will state it only for tensor algebras because we already have the necessary notation. It formalizes the sort of `leading terms' argument that we wish to make in proving Proposition \ref{injection2}.

\begin{lemma} \label{leadingTerms} Let $f: T(A) \longrightarrow T(B)$ be a homomorphism of $\Zps$-modules (not necessarily of algebras) with $A$ free. Suppose that $p^{s-1}\zeta_j \circ f \circ \iota_k = 0$ whenever $j<k$ and that for each $k \in \mathbb{N}$, the map $\zeta_k \circ f \circ \iota_k$ is an injection. Then $f$ is also an injection. \end{lemma}

\begin{proof} Firstly, since $T(A)$ is a free $\Zps$-module, it suffices to show that if $f(p^{s-1} x) = 0$, for $x \in T(A)$, then $p^{s-1} x = 0$. This is precisely showing injectivity of the restriction of $f$ to $p^{s-1}T(A)$. The module $T(A)$ is filtered by the submodules $\bigoplus_{i=1}^k A^{\otimes i}$ for $k \in \mathbb{N}$, so it further suffices to show that each map $$\zeta_{\leq k} \circ f \circ \iota_{\leq k} : p^{s-1} \bigoplus_{i=1}^k A^{\otimes i} \longrightarrow p^{s-1} \bigoplus_{i=1}^k B^{\otimes i}$$ is injective.

We proceed by induction. The case $k=1$ is immediate, so assume that the result is known for $k-1$. Write $\bigoplus_{i=1}^k A^{\otimes i} \cong \bigoplus_{i=1}^{k-1} A^{\otimes i} \oplus A^{\otimes k}$, so that $\iota_{\leq k}$ is identified with $\iota_{\leq (k-1)} \oplus \iota_k$. Suppose that $f(y)=0$ for $y \in p^{s-1} \bigoplus_{i=1}^k A^{\otimes i}$, so that there exists $x \in \bigoplus_{i=1}^k A^{\otimes i}$ with $y=p^{s-1}x$. We must show that $y=0$. Write $x = x' + x_k$, for $x' \in \bigoplus_{i=1}^{k-1} A^{\otimes i}$ and  $x_k \in A^{\otimes k}$. Now, $$\zeta_{\leq (k-1)} \circ f \circ \iota_{\leq k} (y) = p^{s-1} \zeta_{\leq (k-1)} \circ f \circ \iota_{\leq k}(x) $$ $$= p^{s-1} \zeta_{\leq (k-1)} \circ f( \iota_{\leq (k-1)}x' + \iota_k(x_k)) = \zeta_{\leq (k-1)} \circ f(p^{s-1} \iota_{\leq (k-1)}x'),$$ since $p^{s-1}\zeta_j \circ f \circ \iota_k = 0$ for $j<k$. By inductive hypothesis, this implies that $p^{s-1} x' = 0$, so $y = x_k$, and $$\zeta_k \circ f \circ \iota_{\leq k} (y) = \zeta_k \circ f \circ \iota_{ k} (p^{s-1} x_k).$$

By assumption, $\zeta_k \circ f \circ \iota_{ k}$ is an injection, so $p^{s-1} x_k = 0$, and therefore $y=0$, as required. \end{proof}

\begin{proposition} \label{injection2} Let $X$ be a connected $CW$-complex, let $p$ be an odd prime, and let $s \leq r \in \mathbb{N}$. Let $\mu: P^{n+1}(p^r) \longrightarrow \Sigma X$ be a continuous map. If the induced map $$\mu_* : H_*(P^{n+1}(p^r);\Zps) \longrightarrow H_*(\Sigma X;\Zps)$$ is an injection, then $$(\Omega \mu)_* : H_*(\Omega P^{n+1}(p^r);\Zps) \longrightarrow H_*(\Omega \Sigma X;\Zps)$$ is also an injection. \end{proposition}

The principal difficulty in the proof is that $\textrm{Im}(\mu_*)$ might not be contained in the tensor algebra $T(\widetilde{H}_*( X;\Zps))$ inside $\widetilde{H}_*(\Omega \Sigma X ; \Zps)$. We navigate this using the condition $p^{s-1} C = 0$ of Proposition \ref{BSGen}, which prevents the complementary part $C$ from interfering too much. This proposition is much simpler to prove if one assumes that the map $\mu$ is a suspension, but this assumption is not necessary.

\begin{proof} Homology is taken with $\Zps$-coefficients throughout. By the Bott-Samelson theorem (Theorem \ref{BS}), we have an isomorphism $$\widetilde{\eta_*} : T(\widetilde{H}_*(P^n(p^r))) \longrightarrow H_*(\Omega P^{n+1}(p^r)),$$ so it suffices to show that $(\Omega \mu)_* \circ \widetilde{\eta_*}$ is an injection. By definition, $(\Omega \mu)_* \circ \widetilde{\eta_*}$ is the unique map of algebras extending $$(\Omega \mu)_* \circ \eta_* : \widetilde{H}_*(P^n(p^r)) \longrightarrow H_*(\Omega \Sigma X),$$ and by the triangle identities for the adjunction $\Sigma \dashv \Omega$, we have that $(\Omega \mu) \circ \eta = \overline{\mu}$. Thus, $(\Omega \mu)_* \circ \widetilde{\eta_*}$ is the unique map of algebras extending $\overline{\mu}_*$.

The other triangle identity tells us that we have a commuting diagram \begin{center}
\begin{tabular}{c}
\xymatrix{
P^{n+1}(p^r) \ar^{\Sigma\overline{\mu}}[r] \ar_{\mu}[dr] & \Sigma \Omega \Sigma X \ar^{\ev}[d] \\
& \Sigma X.
}
\end{tabular}
\end{center} By assumption, $\mu$ induces an injection on homology, so $\ev \circ (\Sigma \overline{\mu})$ must also induce an injection on homology.

The next step is to turn the problem into one about tensor algebras. Proposition \ref{BSGen} gives a module decomposition of $\widetilde{H}_*(\Omega \Sigma X)$ as the direct sum $T(\widetilde{H}_*(X)) \oplus C$ with $p^{s-1} C = 0$. Under this decomposition, the inclusion associated to the factor $T(\widetilde{H}_*(X))$ is $\widetilde{\eta_*}$. Write $\splitting$ for the projection. Consider the diagram \begin{center}
\begin{tabular}{c}
\xymatrix{ 
& T(\widetilde{H}_*(X)) \ar@/^/^{\sigma \circ \widetilde{\eta_*}}[d] \\
\widetilde{H}_*(P^{n+1}(p^r)) \ar^{(\Sigma\overline{\mu})_*}[r] \ar_{\mu_*}[dr] & \widetilde{H}_*(\Sigma \Omega \Sigma X) \ar^{\ev_*}[d] \ar@/^/^{\splitting \circ \sigma^{-1}}[u] \\
& \widetilde{H}_*(\Sigma X).
}
\end{tabular}
\end{center}

The maps $\sigma \circ \widetilde{\eta_*}$ and $\splitting \circ \sigma^{-1}$ differ from $\widetilde{\eta_*}$ and $\splitting$ only up to suspension isomorphisms, so they are the inclusion and projection associated to the decomposition of $\widetilde{H}_*(\Sigma \Omega \Sigma X)$ obtained by suspending that of Proposition \ref{BSGen}. Lemma \ref{factorTensor} (with $g = \ev_*$, $f = (\Sigma \overline{\mu})_*$, $i_A = \sigma \circ \widetilde{\eta_*}$, and $\pi_A = \splitting \circ \sigma^{-1}$) then tells us that the whole composite $\ev_* \circ (\sigma \circ \widetilde{\eta_*}) \circ (\splitting \circ \sigma^{-1}) \circ (\Sigma \overline{\mu})_*$ is an injection. Furthermore, by Lemma \ref{evproj}, the composite $\ev_* \circ (\sigma \circ \widetilde{\eta_*})$ is identified via suspension isomorphisms with the projection $\zeta_1 : T(\widetilde{H}_*(X)) \longrightarrow \widetilde{H}_*(X)$, so the composite $\zeta_1 \circ \splitting \circ \overline{\mu}_*$ is an injection.

Let $a$ and $b$ form a basis of the free $\Zps$-module $\widetilde{H}_*(P^{n}(p^r))$. By Lemma \ref{summand}, the images of $a$ and $b$ under $\zeta_1 \circ \splitting \circ \overline{\mu}_*$ generate a summand isomorphic to $(\Zps)^2$ inside $\widetilde{H}_*(X)$.

Since $\widetilde{H}_*(P^{n}(p^r))$ is free on $a$ and $b$, a basis of $T(\widetilde{H}_*(P^{n}(p^r)))$ consists of the elements $x_1 \otimes \dots \otimes x_k$, for $k \in \mathbb{N}$, where each $x_i$ is equal to $a$ or $b$. We will show that the image of this basis under $(\Omega \mu)_* \circ \widetilde{\eta_*}$ is the basis of a free $\Zps$-submodule of $H_*(\Omega \Sigma X)$, which will imply the result. Firstly, since $(\Omega \mu)_* \circ \widetilde{\eta_*}$ is the unique map of algebras extending $\overline{\mu}_*$, we have

$$p^{s-1} \zeta_j \circ \splitting \circ (\Omega \mu)_* \circ \widetilde{\eta_*} (x_1 \otimes \dots \otimes x_k) = p^{s-1} \zeta_j \circ \splitting (\overline{\mu}_*(x_1) \otimes \dots \otimes \overline{\mu}_*(x_k))$$ $$= \zeta_j \circ \splitting ( p^{s-1} (\overline{\mu}_*(x_1) \otimes \dots \otimes \overline{\mu}_*(x_k))).$$

By Proposition \ref{BSGen}, we may write each $\overline{\mu}_*(x_i)$ as $\widetilde{\eta_*}(t_i)+c_i$, for $t_i = \splitting(\overline{\mu}_*(x_i)) \in T(\widetilde{H}_*(X))$ and some $c_i$ with $p^{s-1} c_i = 0$. The above is therefore equal to $$\zeta_j \circ \splitting ( p^{s-1} ((\widetilde{\eta_*}(t_1)+c_1) \otimes \dots \otimes (\widetilde{\eta_*}(t_k)+c_k)) = \zeta_j \circ \splitting ( p^{s-1} (\widetilde{\eta_*}(t_1) \otimes \dots \otimes \widetilde{\eta_*}(t_k))$$ $$= \zeta_j ( p^{s-1} t_1 \otimes \dots \otimes t_k) = \begin{cases} p^{s-1} \zeta_1(t_1) \otimes \dots \otimes \zeta_1(t_k) & j=k \\ 0 & j<k \end{cases}$$ by Lemma \ref{leadingTrims}. Since $t_i = \splitting(\overline{\mu}_*(x_i))$, we have $$p^{s-1} \zeta_1(t_1) \otimes \dots \otimes \zeta_1(t_k) = p^{s-1}  (\zeta_1 \circ \splitting(\overline{\mu}_*(x_1))) \otimes \dots \otimes (\zeta_1 \circ \splitting(\overline{\mu}_*(x_k))).$$

Now, each $x_i$ is equal to $a$ or $b$, and we have seen that the images of $a$ and $b$ under $\zeta_1 \circ \splitting \circ \overline{\mu}_*$ generate a $(\Zps)^2$-summand inside $\widetilde{H}_*(X)$. It follows that the elements $\zeta_1 \circ \splitting(\overline{\mu}_*(x_1)) \otimes \dots \otimes \zeta_1 \circ \splitting(\overline{\mu}_*(x_k))$ generate a copy of $T((\Zps)^2)$ inside $T(\widetilde{H}_*(X))$.

The above calculation therefore tells us that the map $\zeta_k \circ \splitting \circ (\Omega \mu)_* \circ \widetilde{\eta_*} \circ \iota_k$ carries $p^{s-1}$ times a basis of $\widetilde{H}_*(P^{n}(p^r))^{\otimes k} \subset T(\widetilde{H}_*(P^{n}(p^r)))$ to $p^{s-1}$ times a basis of $((\Zps)^2)^{\otimes k} \subset T((\Zps)^2)$ inside $T(\widetilde{H}_*(X))$. This implies that the restriction of $\zeta_k \circ \splitting \circ (\Omega \mu)_* \circ \widetilde{\eta_*} \circ \iota_k$ to $p^{s-1} \widetilde{H}_*(P^{n}(p^r))^{\otimes k}$ is an injection, so $\zeta_k \circ \splitting \circ (\Omega \mu)_* \circ \widetilde{\eta_*} \circ \iota_k$ must itself be an injection and we have also seen that $p^{s-1}\zeta_j \circ \splitting \circ (\Omega \mu)_* \circ \widetilde{\eta_*} \circ \iota_k = 0$ for $j<k$

Thus, by Lemma \ref{leadingTerms}, $\splitting \circ (\Omega \mu)_* \circ \widetilde{\eta_*}$ is an injection, so $(\Omega \mu)_* \circ \widetilde{\eta_*}$ is an injection, as required. \end{proof}

\section{Proof of Theorems \ref{HPrelim} and \ref{HCriterion}} \label{proof2Section}

In this section we will prove Theorem \ref{HPrelim}, and then from that, together with Proposition \ref{injection2}, deduce Theorem \ref{HCriterion}.

\begin{proof}[Proof of Theorem \ref{HPrelim}] By Lemma \ref{toTheFields} it suffices to prove the theorem when $t=s$. Combining Lemmas \ref{PhiPiDiffs}, \ref{hurewiczTriangle}, and naturality of the maps $\beta^r$, $\rho^s$, and $h$ with respect to the map of spaces $\Omega \mu$, we obtain the following commuting diagram: \begin{center}
\begin{tabular}{c}
\xymatrix{
L'(x,dx) \ar_{d}[dd] \ar^(0.4){\PhiPi^{r,r}}[r] & \pi_*(\Omega P^{n+1}(p^r); \Zpr) \ar^{\beta^r}[d] \ar^(0.57){(\Omega \mu)_*}[r] & \pi_*(\Omega Y; \Zpr) \ar^{\beta^r}[d]  \\
& \pi_*(\Omega P^{n+1}(p^r)) \ar^{\rho^s}[d] \ar^(0.57){(\Omega \mu)_*}[r] & \pi_*(\Omega Y) \ar^{\rho^s}[d]  \\
L'(x,dx) \otimes \Zps \ar^(0.45){\PhiPi^{r,s}}[r] \ar_{\quot}[d] & \pi_*(\Omega P^{n+1}(p^r); \Zps) \ar^{h}[d] \ar^(0.57){(\Omega \mu)_*}[r] & \pi_*(\Omega Y; \Zps) \ar^{h}[d] \\
L(x,dx) \otimes \Zps \ar^(0.45){\PhiH^{r,s}}[r] & H_*(\Omega P^{n+1}(p^r); \Zps) \ar^(0.57){(\Omega \mu)_*}[r] & H_*(\Omega Y; \Zps). \\
}
\end{tabular}
\end{center}

By Corollary \ref{combining}, $\textrm{Im}((\Omega \mu)_* \circ \PhiH^{r,s} \circ \quot \circ  d)$ is $\Zps$-hyperbolic. By commutativity of the diagram, $(\Omega \mu)_* \circ \PhiH^{r,s} \circ \quot \circ  d = h \circ \rho^s \circ (\Omega \mu)_* \circ \beta^r \circ \PhiPi^{r,r}$, so the image of the latter map is also $\Zps$-hyperbolic.

We thus obtain a diagram \begin{center}
\begin{tabular}{c}
\xymatrix{
\pi_*(\Omega Y) \ar^{h \circ \rho^s}[dr] & \\
\textrm{Im}((\Omega \mu)_* \circ \beta^r \circ \PhiPi^{r,r}) \ar^(.4){h \circ \rho^s}[r] \ar@{^{(}->}[u] & \textrm{Im}((h \circ \rho^s) \circ ((\Omega \mu)_* \circ \beta^r \circ \PhiPi^{r,r})).
}
\end{tabular}
\end{center}

The bottom map is a surjection by choice of codomain, and we have shown above that this codomain is $\Zps$-hyperbolic. The domain of $(\Omega \mu)_* \circ \beta^r \circ \PhiPi^{r,r}$ is $L'(x,dx)$, which is a $\Zpr$-module, hence is automatically annihilated by multiplication by $p^r$. Therefore, the group in the bottom left, $\textrm{Im}((\Omega \mu)_* \circ \beta^r \circ \PhiPi)$, is also annihilated by multiplication by $p^r$. The group in the bottom right, $\textrm{Im}((\Omega \mu)_* \circ \beta^r \circ \PhiPi)$, is contained in $H_*(\Omega Y; \Zps)$, hence is annihilated by multiplication by $p^s$. This means that we can apply Corollary \ref{surjectOnHype} (The `Sandwich' Lemma) to see that $\pi_*(\Omega Y) \cong \pi_{*+1}(Y)$ is $p$-hyperbolic concentrated in exponents $s, s+1, \dots, r$, so by definition $Y$ is $p$-hyperbolic concentrated in exponents $s, s+1, \dots, r$, which completes the proof. \end{proof}

Theorem \ref{HCriterion} now follows.

\begin{proof}[Proof of Theorem \ref{HCriterion}] By Proposition \ref{injection2}, $(\Omega \mu)_*$ is an injection, so by Theorem \ref{HPrelim}, $\Sigma X$ is $p$-hyperbolic concentrated in exponents $s, s+1, \dots, r$, as required. \end{proof}

\printbibliography

\end{document}